\documentclass[a4paper,12pt,twoside]{amsart}
\usepackage{amsmath,amsthm,amsfonts}

\newtheorem{thm}{Theorem}[section]

\newtheorem{lem}[thm]{Lemma}

\newtheorem{proposition}[thm]{Proposition}

\newtheorem{nota}[thm]{Notation}

\theoremstyle{definition}

\newtheorem{ex}[thm]{Example}

\newtheorem{rem}[thm]{Remark}
\newtheorem{rems}[thm]{Remarks}

\newcommand{\R}{{\mathbb{R}}}

\newcommand{\E}{{\mathbb{E}}}

\newcommand{\Z}{{\mathbb{Z}}}
\newcommand{\N}{{\mathbb{N}}}
\newcommand{\T}{{\mathbb{T}}}
\newcommand{\PP}{{\mathbb{P}}}

\newcommand{\Cal}{\mathcal}
\newcommand{\cal}{\mathcal}

\def \mod {{\rm \ mod \,} }

\def \for {{\rm\ for \ }}
\def \and {{\rm\ and \ }}

\def\e{{\rm e}}

 \def\k{{\underline k}}
 \def\el{{\underline \ell}}
\def\t{{\underline t}} 
\def\0{{\underline 0}} 
 
 \def\p {{\underline p}}
 \def\r{{\underline r}}

\def\stm0{{\setminus \{\0\}}}

\def\eop{\qed}
\def\Proof {\vskip -3mm {{\it Proof}. }}
\def\proof {\vskip -3mm {{\it Proof}. }}

\def \Log {{\rm Log}}
\def \Card {{\rm Card}}

\begin{document}

\date{\today}
\parskip=2mm
\baselineskip 15pt
\parindent=0mm

\title[Empirical process sampled along a stationary process] {Empirical process \\ sampled along a stationary process}


\author{Guy Cohen and Jean-Pierre Conze}
\address{Guy Cohen, \hfill \break School of Electrical Engineering, \hfill \break Ben-Gurion University, Israel} \email{guycohen@bgu.ac.il}
\address{Jean-Pierre Conze, \hfill \break IRMAR, CNRS UMR 6625, \hfill \break University of Rennes, Campus de Beaulieu, 35042 Rennes Cedex, France} \email{conze@univ-rennes1.fr}

\maketitle

\begin{abstract}  
Let $(X_{\el})_{\el \in \Z^d}$ be a real random field (r.f.) indexed by $\Z^d$ with common probability distribution function $F$. 
Let $(z_k)_{k=0}^\infty$ be a sequence in $\Z^d$. The empirical process obtained by sampling the random field along $(z_k)$ is
$\sum_{k=0}^{n-1} [{\bf 1}_{X_{z_k} \leq s}- F(s)]$. 

We give conditions on $(z_k)$ implying the Glivenko-Cantelli theorem for the empirical process sampled along $(z_k)$
in different cases (independent, associated or weakly correlated random variables).  We consider also the functional central limit theorem when the $X_\el$'s are i.i.d.

These conditions are examined when $(z_k)$ is provided by an auxiliary stationary process in the framework of ``random ergodic theorems''.
\end{abstract}

\tableofcontents

\section*{\bf Introduction}

For a sequence $(X_k)$ of real i.i.d. random variables with common probability distribution function $F$, 
the empirical process is defined by $\sum_{k=0}^{n-1} \, [{\bf 1}_{X_k \leq s}- F(s)]$. Recall two classical results.
\hfill \break
(A) the Glivenko-Cantelli theorem: 
\hfill \break
{\it a.s. the sequence of empirical distribution functions $F_n(s) := \frac1n \sum_{k=0}^{n-1} {\bf 1}_{X_k \leq s}$ 
converges uniformly to $F$, i.e. $\sup_s |F_n(s) - F(s)| \to 0$; }
\hfill \break
(B) a functional central limit theorem (FCLT): if the r.v.s $X_k$ have a common distribution $F$ over $[0, 1]$, then
\hfill \break
{\it the process ${1 \over \sqrt n} \sum_{k=0}^{n-1} \, [{\bf 1}_{X_k \leq s}- F(s)]$ converges weakly to a Brownian bridge in the space of cadlag functions on $[0, 1]$.}

In this paper we study the extension of these results when the process is sampled along a subsequence, analogously to what is done for limit theorems in random scenery. 

In the sequel, for $d \geq1$, $(X_{\el})_{\el \in \Z^d}$ will be a real random field (r.f.) indexed by $\Z^d$  
defined on a probability space $(\Omega, \mathcal F,\PP)$ with common probability distribution function $F$. 
The expectation on $(\Omega, \PP)$ is denoted by $\E$.
We consider in particular the case of a r.f. of i.i.d. r.v.'s or of stationary associated r.v.'s.

Let $(z_k)_{k=0}^\infty$ be a sequence in $\Z^d$. The process obtained by sampling the random field along $(z_k)$ is
$W_n(s) := \sum_{k=0}^{n-1} [{\bf 1}_{X_{z_k} \leq s}- F(s)]$.  

We will call $W_n(s)$ ``empirical process sampled along $(z_k)$'', or simply ``sampled empirical process''.
A general question is whether the above results (A), (B) extend to the sampled empirical process $W_n(s)$, 
in particular when $(z_k)$ is given by another stationary process with values in $\Z^d$. 

In Section \ref{sectGen}, we give conditions on $(z_k)$ implying that (A) and (B) are still valid for an empirical process sampled along $(z_k)$
in different cases: independent, associated or weakly correlated random variables.
The conditions are expressed in terms of the following quantities associated to the sequence $(z_k)$ in $\Z^d$:
local time, maximal local time and number of self-intersections (up to time $n$) defined, for $n \geq 1$, by
\begin{eqnarray}
&&N_n(\el):=\#\{0\le k\le n-1:\ z_k=\el\}, \nonumber \\ 
&&\ M_n:= \max_{\el}N_n(\el), \ V_n:=\#\{0\le j,k \le n-1 :\, z_j=z_k\}. \label{notaVn}
\end{eqnarray}
They satisfy $\sum_\el  N_n(\el) = n$ and $n \leq V_n=\sum_\el N_n^2(\el) \leq n M_n \leq n^2$.

In the other sections, $(z_k)$ is given by a stationary process (or equivalently by the sequence $(S_k f(x))_{k \geq 1}$ of ergodic sums of a function $f$ over a dynamical system).

The conditions found in Section \ref{sectGen} lead to study the local times, maximum number of visits, number of self-intersections for the sequence $(S_k f(x))$. 
General remarks are presented in Section \ref{genCocy}. Then in Section \ref{exSect}, we consider two families of examples: random walks and some ergodic sums over a rotation.

The Glivenko-Cantelli theorem along ergodic sums (extension of (A)) is strongly related to random ergodic theorems, 
in particular to results in \cite{LPWR94} and \cite{LLPVW02}. This is discussed in the last Section \ref{sectGC0}.

Finally let us mention the quenched FCLT for the 2-parameters process 
$$W_n(s, t) := \sum_{k=0}^{{[n t]}-1} \, [{\bf 1}_{X_{Z_k(x)}  \leq s}- F(s)], \, (s, t) \in [0,1]^2.$$
When $(X_\el)$ is a r.f. of i.i.d. r.v.'s indexed by $\Z^2$ and when the sampling is provided by a 2-dimension centered random walk $(Z_k)$ with a moment of order 2, 
the weak convergence for a.e. $x$ toward a Kiefer-M\"uller process can be shown. This will be the content of a forthcoming paper.

\vskip 3mm 
{\bf Acknowledgements.} Part of this research was done during visits of the first author to the IRMAR at the University of Rennes 1 
and of the second author to the Center for Advanced Studies in Mathematics at Ben Gurion University. 
The authors are grateful to their hosts for their support.

\vskip 3mm
\goodbreak
\section{\bf General results on the empirical process along a sub-sequence} \label{sectGen}

\subsection{Preliminaries}

In this subsection, results on the empirical process along a sub-sequence are shown for independent variables, 
as well for some of them for wider classes (associated, PDQ and weakly correlated random variables).
We start by recalling some notions and auxiliary results.

\vskip 3mm
{\it 1) Associated variables}

{\bf Definition } (cf. \cite{EPW}): A finite set of real random variables ${\bf T}=(T_1, T_2,\ldots, T_n)$ is said to be {\it associated} if ${\rm Cov} [f({\bf T}), g({\bf T})]\ge0$,
for every coordinate-wise non-decreasing functions $f = f(x_1, ..., x_n)$ and $g = g(x_1, ..., x_n)$ for which $\E[f({\bf T} )]$, $\E[g({\bf T})]$,
$\E[ f( {\bf T}) \, g( {\bf T})]$ exist. An infinite set of random variables is associated if any finite subset of it is associated. 

Association of random variables is preserved under taking subsets and forming unions of independent sets (of associated random variables). 
In particular a family of independent variables is associated. 

Clearly, orthogonal associated random variables are independent. 
Examples of (non independent) stationary associated processes with absolutely summable series of correlations are provided by some Ising models. 
References to such examples of stationary $\Z^d$ random fields which satisfies the FKG inequalities and with absolutely summable correlations 
can be found in Newman's paper \cite{N80}. Notice that the FKG inequalities expresses the association property of the r.v.'s. 

2) {\it PQD variables}

Two r.v.'s $X, Y$ are called (cf. \cite{Leh66}) {\it positively quadrant dependent (PQD)} if, 
$$\PP(X>x, Y>y)\ge \PP(X>x) \, \PP(Y>y), \forall x,y\in \R.$$
The property is preserved by centering.  Any pairwise associated r.v.'s are pairwise PQD.
Pairwise independent random variables are pairwise PQD associated.

Two random variables $X$ and $Y$ are PQD if and only if for every non-decreasing functions $f$ and $g$, $\text{Cov}(f(X), g(Y) )\ge 0$ 
(whenever the covariance exists) (\cite[Theorem 4.4]{EPW}). 

3) We will use the following results:

{\it a) Maximal inequality of Newman and Wright} \cite[Inequality (12)]{NW}:

{\it If $(W_i)$ is a sequence of centered associated, square integrable random variables, it holds:}
\begin{eqnarray}
&&\PP(\max_{1\le j \le n} |\sum_{i=1}^j W_i|\geq \lambda \, \| \sum_{i=1}^n W_i \|_2)
\leq 2 \PP(|\sum_{i=1}^n W_i|\ge (\lambda-\sqrt{2}) \, \| \sum_{i=1}^n W_i \|_2), \forall \lambda \geq 0. \label{NW0}
\end{eqnarray}  
{\it b) Hoeffding's identity} (see \cite[Theorem 3.1]{Am17})

{\it Let $X, Y$ be random variables with finite second moments. For any absolutely continuous functions $f, g$ on $\mathbb R$, 
such that $\E[f^2(X)+g^2(Y)] < \infty$, it holds
$$\text{Cov} (f(X),g(Y))=\int_{-\infty}^\infty \int_{-\infty}^\infty f'(x) g'(y)[\PP(X> x, Y> y) - \PP(X> x)\PP(Y> y ) ] dx dy.$$
In particular, if $X, Y$ are PDQ random variables, if $|f'|,|g'|\leq M$ a.e., we have 
$$|\text{Cov} (f(X), g(Y))|\le M^2\text{Cov}(X,Y).$$}
4) Uniformity in the analogues of Glivenko-Cantelli theorem will follow from the lemma:
\begin{lem} \label{ChungLem} \cite[Lemma, p 140]{Chung} Let $F_n$, $F$ be a family of right continuous distributions on $\mathbb R$.  
Assume that, for each point $x$ in a dense countable set $Q\subset \mathbb R$, we have $F_n(x)\to F(x)$.
Let $J$ be the set of jumps of $F$ and assume that $F_n(x)-F_n(x^-)\to F(x)-F(x^-)$ for every $x\in J$. Then $F_n(x)\to F(x)$ uniformly in $\mathbb R$.
\end{lem}

{\bf A strong law of large numbers}

First we state a law of large numbers for bounded r.v.'s valid under weak hypotheses.

Let $(U_\el)_{\el\in\mathbb Z^d}$ be a r.f. indexed by $\Z^d$ of square integrable r.v's on a probability space $(\Omega, \mathcal F, \PP)$.
Let $(z_k)_{k \geq 0}$ be a sequence in $\Z^d$, $d \geq 1$, with numbers of self-intersections $V_n, n \geq 1$. 
The partial sums along $(z_k)$ are denoted by $\displaystyle S_n :=\sum_{k=0}^{n-1} U_{z_k}$.

By the Cauchy-Schwarz inequality, if $\sum_\el \sup_\r |\langle U_{\r+\el} , U_{\r}\rangle| < +\infty$, if holds for a finite constant $C_0$:
\begin{eqnarray}
&&\|\sum_{i=0}^{n-1} U_{z_i}\|_2^2 
= \sum_\el \sum_\r N_n(\r+\el) N_n(\r) \langle U_{\r+\el} , U_{\r}\rangle \leq  V_n \sum_\el \sup_\r |\langle U_{\r+\el} , U_{\r}\rangle| = C_0 V_n. \label{norm2p}
\end{eqnarray}
In particular if the r.f. is stationary and the series of correlations is absolutely summable (i.e., $\sum_{\el \in \Z^d} |\langle X_\0, X_\el \rangle| < +\infty$),
then the spectral density of the r.f. exists and is the continuous non-negative function $\rho$ on $\T^d$ with Fourier coefficients
$\displaystyle \int_{\T^d} e^{2 \pi i \langle \el, \t \rangle} \, \rho(\t) \, d\t = \langle X_\0, X_\el \rangle$ and it holds:
\begin{eqnarray}
&&\|S_n\|_2^2 = \|\sum_{i=0}^{n-1} U_{z_i}\|_2^2 \leq  V_n \sum_\el |\langle U_\el , U_\0\rangle |. \label{norm2}
\end{eqnarray}
\begin{proposition} \label{LLN1} Suppose the r.v.'s $U_\el$ on $(\Omega, \PP)$ centered and uniformly bounded by the same constant $K$, $\|U_\el\|_\infty \leq K, \forall \el$. 
Assume that $(z_k)$ is such that
\begin{eqnarray}
V_n \leq C_1 {n^2 \over (\log n)^\beta}, \text{ for constants } C_1, \beta. \label{condibeta}
\end{eqnarray}
1) Then, if $\beta > 1$ and $\displaystyle \sum_{\el \in \Z^d} \sup_{r \in \Z^d} |\langle U_{\r+\el} , U_{\r}\rangle| < + \infty$,
\hfill \break 2) or if $\beta > \zeta$ for some $\zeta \in [1, 2]$ and the r.f. $(U_\el)$ is stationary with 
$\displaystyle \sum_{\el \in \Z^d} |\langle U_\el , U_{\0}\rangle|^\zeta < \infty$, 
\hfill \break the (strong) LLN holds: $\displaystyle {S_n(\omega) \over n} \to 0$, for $\PP$-a.e $\omega$.
\end{proposition}
\proof 1)  For convenience, if $t$ is in $\R^+$, we define $S_t$ as $S_{[t]}$. From (\ref{norm2p}) it follows
$$\int ({|S_n| \over n})^2 \, d\PP \leq C_0 {V_n \over n^2} \leq C_0 C_1 {1 \over (\log n)^\beta}.$$
Therefore, putting $\beta= 1+ \eta$ and $\alpha = 1 - \eta/2$ (which implies $\alpha \beta > 1$) we have
\begin{eqnarray*}
\sum_k \int ({|S_{2^{k^\alpha}}| 
\over 2^{{k^\alpha}}})^2 \, d\PP \leq C_0 C_1 \sum_k {1 \over  (\log {2^{k^\alpha}})^\beta} = C' \sum_k {1 \over k^{\alpha \beta}} < +\infty;
\end{eqnarray*}
hence: $\displaystyle \lim_{k \to +\infty}{S_{2^{k^\alpha}} \over 2^{{k^\alpha}}} = 0$, a.e.

For $n \geq 1$, let $k_n$ be such that $2^{{(k_n)^\alpha}} \leq n < 2^{{(k_n+1)^\alpha}}$
(that is: $k_n = [(\log_2 n)^{1/\alpha}]$).

We put $q_n:=2^{{(k_n+1)^\alpha}} - 2^{{(k_n)^\alpha}}$ 
and  $p_n = n - 2^{{(k_n)^\alpha}} \leq q_n$. 

For $q_n$, the following estimate holds: 
$\displaystyle q_n=2^{{(k_n)^\alpha}} (2^{{(k_n+1)^\alpha} - {(k_n)^\alpha}} - 1) \sim C'' {2^{{(k_n)^\alpha}} \over {(k_n)^{1 -\alpha}}}$.

Using the uniform boundedness of the r.v.'s, we can write:
\begin{eqnarray*}
&&|{S_n \over n} - {S_{2^{{(k_n)^\alpha}}} \over 2^{{(k_n)^\alpha}}}| 
= |{S_{2^{{(k_n)^\alpha}}} + \sum_{i= 2^{{(k_n)^\alpha}}}^{2^{{(k_n)^\alpha}}+p_n} U_{z_i} \over 2^{{(k_n)^\alpha}}+p_n} - {S_{2^{{(k_n)^\alpha}}} \over 2^{{(k_n)^\alpha}}}| 
= |{2^{{(k_n)^\alpha}} \, \sum_{i= 2^{{(k_n)^\alpha}}}^{2^{{(k_n)^\alpha}}+p_n} U_{z_i} - p_n S_{2^{{(k_n)^\alpha}}} \over 2^{{(k_n)^\alpha}} (2^{{(k_n)^\alpha}}+p_n)}| \\
&&\leq {2^{{(k_n)^\alpha}} \, \sum_{i= 2^{{(k_n)^\alpha}}}^{2^{{(k_n)^\alpha}}+p_n} |U_{z_i}| + p_n |S_{2^{{(k_n)^\alpha}}}| \over 2^{{(k_n)^\alpha}} (2^{{(k_n)^\alpha}} +p_n)}
\leq {2^{{(k_n)^\alpha}} \, \sum_{i= 2^{{(k_n)^\alpha}}}^{2^{{(k_n)^\alpha}}+q_n} |U_{z_i}| + q_n |S_{2^{{(k_n)^\alpha}}}| \over 2^{{(k_n)^\alpha}} (2^{{(k_n)^\alpha}})} \\
&&\leq {q_n K 2^{{(k_n)^\alpha}}  + q_n |S_{2^{{(k_n)^\alpha}}}| \over 2^{{(k_n)^\alpha}} (2^{{(k_n)^\alpha}})}
={q_n \over 2^{{(k_n)^\alpha}}} (K + \, {|S_{2^{{(k_n)^\alpha}}}| \over 2^{{(k_n)^\alpha}}}).
\leq {C \over 2 {(k_n)^{1 - \alpha}}} (K + \, {|S_{2^{{(k_n)^\alpha}}}| \over 2^{{(k_n)^\alpha}}}) \to 0. 
\end{eqnarray*}

2) We consider now the stationary case. Since $\zeta =1$ is special case of 1), we assume $\zeta \in ]1, 2]$. 
We put $\beta = \zeta + \eta$, where $\eta$ is $> 0$ in view of the hypothesis.

First, suppose that $\zeta =2$. Then under the hypothesis, the r.f. has a spectral measure $\nu_\varphi$ absolutely continuous 
with respect to the Lebesgue measure $\lambda$ on the torus with a density $\rho \in L^2(d\t)$ 
given by the Fourier series $\rho(\t) = \sum_{\el \in \Z^d} \langle U_\el , U_{\0} \rangle \, \e^{2 i \pi \langle \el, \t \rangle}$.

Using the inequality $\lambda \{\rho > M_n\} \leq M_n^{-2} \, \|\rho\|_2^2$, we can write:
\begin{eqnarray*}
&&{\|S_n\|_2^2 \over n^2} =\frac1{n^2} \, \int_{\T^d} |\sum_{j=0}^{n-1} e^{2 \pi i \langle z_j, \t \rangle}|^2 \, d \nu_\varphi(\t) 
\leq \frac{M_n}{n^2} \, \int_{\T^d} |\sum_{j=0}^{n-1} e^{2 \pi i \langle z_j, \t \rangle}|^2 \, d\t+ \int_{\rho > M_n} \, \rho \, d\t \\
&\leq& M_n {V_n \over n^2}+(\lambda \{\rho > M_n\})^\frac12 \|\rho\|_2 \leq M_n {V_n \over n^2} + M_n^{-1} \, \|\rho\|_2^2.
\end{eqnarray*}
Taking $M_n = (\log n)^{1 +  \frac12 \eta}$, we obtain the bound 
$$\frac1{n^2} \, \|\sum_{j=0}^{n-1} U_{z_j}\|_2^2 \, \leq {C \over (\log n)^{1 + \frac12 \eta}}$$ 
and then we finish the proof as in 1). 

Now, suppose that $\displaystyle \sum_{\el \in \Z^d} |\langle U_\el , U_{\0}\rangle|^\zeta < \infty$ with $1 < \zeta < 2$.
The spectral density $\rho$ exists and is in $L^2(\lambda)$, since $\displaystyle \sum_{\el \in \Z^d} |\langle U_\el , U_{\0}\rangle|^2 < \infty$.
Moreover it belongs to $L^{\zeta'}(\lambda)$ where $\zeta, \zeta'$ are conjugate exponents (see: \cite{Z68}, p. 102, or \cite{HeRo} Th. 31.22),
and it satisfies:
$$\|\rho\|_{\zeta'}\leq (\sum_{\el \in \Z^d} |\langle U_\el, U_0 \rangle|^\zeta)^{1/\zeta}.$$
Hölder's inequality implies:
$\int_{\rho > M_n} \, \rho \, d\t \leq (\lambda \{\rho > M_n\})^{1/\zeta} \|\rho\|_{\zeta'}$.
As 
$$\lambda \{\rho > M_n\} \leq M_n^{-\zeta'} \int \rho^{\zeta'} \, d\t= M_n^{-\zeta'} \|\rho\|_{\zeta'}^{\zeta'},$$
it follows:
$$\int_{\rho > M_n} \, \rho \, d\t \leq  M_n^{-\zeta'/\zeta} \|\rho\|_{\zeta'}^{1+\zeta'/\zeta}.$$
Therefore, we obtain
\begin{eqnarray*}
&&\frac1{n^2} \, \int_{\T^d} |\sum_{j=0}^{n-1} e^{2 \pi i \langle z_j, \t \rangle}|^2 \, d \nu_\varphi(\t) 
\leq M_n {V_n \over n^2} + \int_{\rho > M_n} \, \rho \, d\t \leq M_n {V_n \over n^2} + M_n^{-\zeta'/\zeta} \|\rho\|_{\zeta'}^{1+\zeta'/\zeta}.
\end{eqnarray*}
Now we take $M_n$ such that : $M_n/ (\log n)^\beta =  M_n^{-\zeta'/\zeta}$, i.e. $M_n = (\log n)^{\beta/\zeta'}$.
We get  
$$\frac1{n^2} \, \|\sum_{j=0}^{n-1} U_{z_j}\|_2^2 \, \leq {C \over (\log n)^{\beta(1 - 1/ \zeta')}} 
= {C \over (\log n)^{\beta/\zeta}} = {C \over (\log n)^{1 + \eta/\zeta}} \text{ with } \eta > 0,$$
and the end of the proof is as above. \eop

\begin{rems} 1) Let us give an example of a non stationary r.f. $(U_\el)$ which satisfies Condition 1) of the previous proposition.

We take $(U_\el = V_\el \, W_\el, \el \in \Z^d)$, where $(V_\el)$ and $(W_\el)$ are two r.f.'s independent from each other, 
with $(V_\el)$ centered stationary and such that $\sum_{\el \in \Z^d} |\langle V_\el , V_{\0}\rangle| < \infty$,
and $(W_\el)$ satisfying $\sup_{\el, \p} |\langle W_{\el+\p}, W_{\el}\rangle| < \infty$.

The r.f. $(W_\el)$ can be viewed as a (multiplicative) noise (which can be non stationary) independent from the r.f. $(U_\el)$. Clearly the condition in 1) is satisfied.

2) For a stationary r.f. $(U_\el)$ with a bounded spectral density (but with a series of correlations which may be not absolutely summable), 
then like in 1) the condition $\beta > 1$ is sufficient for the conclusion of the theorem. 
\end{rems}
Now, we give a pointwise bound for the sampled sums, first for i.i.d. r.v.'s, then for a stationary random field $(U_\el)_{\el\in \Z^d}$ of associated r.v.'s. 
\begin{proposition} \label{LIL1} 
1) Suppose that the r.v.'s $U_\el, \el\in \Z^d$, are i.i.d., centered, uniformly bounded by a constant $K$, $\|U_\0\|_\infty \leq K$, and that $\E|U_0|^2=1$. 
Then it holds
\begin{eqnarray}
\limsup_n \frac{|S_n|}{\sqrt{V_n} \, (2 \log \log n)^\frac12}\le K, \, \PP{\text -a.e.} \label{lil0}
\end{eqnarray}
If $V_n = o(n^2 \, (\log\log n)^{-1})$, then $\displaystyle \lim_n \frac {S_n} n =0$, $\PP$-a.e. 

2) Suppose the random field stationary and the r.v.'s $U_\el$ centered associated. 

a) For all $\varepsilon > 0$, it holds, with $\sigma_n :=\|\sum_{i=0}^{n-1} U_{z_i}\|_2$:
\begin{eqnarray}
\limsup_n \frac{|S_n|}{\sigma_n \, ( \log\sigma_n)^{\frac12+\varepsilon}}\le 1, \, \PP{\text -a.e.} \label{casea}
\end{eqnarray} 

b) If moreover the r.f. has a summable series of correlations, then, for all $\varepsilon > 0$,
\begin{eqnarray}
|S_n| = O(\sqrt{V_n} \, (\log n)^{\frac12 + \varepsilon}), \, \PP \text{-a.e.} \label{lil00}
\end{eqnarray}
If $\displaystyle V_n \leq C n^2 \, (\log n)^{-(1+\eta)}) \text{ for some constants } C, \eta > 0$, then
$\displaystyle \lim_n \frac{S_n}n = 0, \, \PP\text{-a.e.}$ 
\end{proposition}
\begin{proof} {\it A)} Recall that $\sigma_n = \|\sum_{i=0}^{n-1} U_{z_i}\|_2$. 
In case 2) we may assume $\|U_{\0}\|_2=1$, and then
in all cases $\sigma_n \leq n$ and by association $\sigma_n\ge   n^\frac12$.
We have in case 1) $\sigma_n = \sqrt{V_n}$ and in case 2b), for associated variables, 
by (\ref{norm2}): $\sigma_n  \leq (\sum_\p \langle U_\p , U_{\0}\rangle)^{\frac12} \, \sqrt{V_n}$.
By association, $\sigma_n$ is non-decreasing and tends to infinity.

For $\rho>1$, let $n_k = n_k(\rho)$ be a strictly increasing sequence of integers such that $\rho^k < \sigma_{n_k} \le \rho^{k+1}$.
Since $1 \leq \sigma^2_{k+1} - \sigma^2_k \leq 1 + 2k$, such a sequence exists after a certain rank. By the choice of $(n_k)$ we have 
\begin{eqnarray}
\rho^k < \sigma_{n_k} \le \rho^{k+1} < \sigma_{n_{k+1}} \leq \rho^{k+2}. \label{rhok}
\end{eqnarray}
Moreover, we have $\sigma_{n_{k+1}}/ \sigma_{n_k} \leq \rho^2$ and, since $\sigma_n \leq n$, $n_k \geq \rho^k$. 
Let $(\lambda_n)$ be a non decreasing sequence of positive numbers such that 
\begin{eqnarray}
&\lambda_{n_k} > \sqrt 2, \ \limsup_k \lambda_{n_{k+1}} / \lambda_{n_k} \le1, \nonumber \\
&\sum_k \PP\big (\big |\sum_{i=0}^{n_k-1}U_{z_i} \big| \geq \, (\lambda_{n_k} - \sqrt 2) \,\|\sum_{i=0}^{n_k-1} U_{z_i} \|_2 \big) < \infty. \label{condSer}
\end{eqnarray}
By the previous inequalities and by Newman-Wright's inequality (\ref{NW0}) for the sequence of centered associated random variables
\footnote{as it is a subset of a set of associated r.v.'s} $(W_i) = (U_{z_i})$, we have
$$\sum_k \PP(\max _{0 \leq j \leq n_{k }-1} \big|\sum_{i=0}^{j} U_{z_i} \big| \geq \lambda_{n_k} \, \|\sum_{i=0}^{n_{k}-1} U_{z_i} \|_2)
\leq 2 \sum_k \PP(|\sum_{j=0}^{n_{k}-1} U_{z_j} |\geq (\lambda _{n_k} - \sqrt{2})\| \sum_{j=0}^{n_{k}-1} U_{z_j} \|_2) < +\infty.$$
By the Borel-Cantelli lemma, it follows:
$$\limsup_k {\max _{0\le j\le n_{k+1}-1} \big|\sum_{i=0}^{j}U_{z_i} \big| \over \, \lambda_{n_{k+1}} \, \sigma_{n_{k+1}}} \leq 1, \PP\text{-a.e.}$$
Hence $\PP$-a.e.
\begin{eqnarray}
&&\limsup_k\frac{\max_{0 \le j <n_{k+1}-1}|\sum_{i=0}^{j} U_{z_i}|} {\lambda_{n_k} \sigma_{n_k}} 
\leq \limsup_k \bigl(\frac{\lambda_{n_{k+1}}} {\lambda_{n_k}} \frac{\sigma_{n_{k+1}}} {\sigma_{n_{k}}}\bigr) \leq \rho^2. \label{rho2}
\end{eqnarray}

Observe that, if $|S_i| > \rho^2 \lambda_i \sigma_i$, for some $i \in [n_k, n_{k+1}[$, then $\max_{0 \leq j<n_{k+1}}|S_j| > \rho^2 \lambda_{n_k} \sigma_{n_k}$. 
This shows: 
\begin{eqnarray*}
\{|S_n|> \rho^2 \lambda_{n} \sigma_n, \, \text{i.o.} \}\subset \{\max_{0 \leq j < n_{k+1}}|S_j| > \rho^2\lambda_{n_k} \sigma_{n_k},\, \text{i.o.} \}.
\end{eqnarray*}
By this inclusion and (\ref{rho2}) it follows: $\displaystyle \limsup_n \frac{|\sum_{i=0}^{n-1} U_{z_i}|} {\lambda_n \sigma_n} \leq \rho^2, \, \PP\text{-a.e.}$

Taking $\rho =\rho_n$ with $\rho_n \downarrow 1$, we obtain
\begin{eqnarray}
\displaystyle \limsup_n \frac{|\sum_{i=0}^{n-1} U_{z_i}|} {\lambda_n \sigma_n} \leq 1, \, \PP\text{-a.e.} \label{limsup0}
\end{eqnarray}

{\it B) Choice of a sequence $(\lambda_k)$ such that (\ref{condSer}) is satisfied.}

{\it Case 1)} 

Suppose that the $U_k$'s are i.i.d. r.v.'s.
Recall that if $(W_j, j \geq 1)$ are centered bounded sequence of independent random variables on $(\Omega, \PP)$, 
for any finite sum of the $W_j$'s it holds by Hoeffding's inequality for differences of martingale (\cite{H63}), for every $\varepsilon > 0$:
\begin{eqnarray}
\PP(|\sum_j W_j|>\varepsilon) \le 2 \exp(- \frac12 \frac{\varepsilon^2}{\sum_j \|W_j\|_\infty^2}). \label{Hoeffding0}
\end{eqnarray} 
We apply it to the family $(N_n(\el) U_\el, \, \el \in \Z^d)$. From the hypotheses, we have:
$$\sum_\el \|N_n(\el) U_\el\|_\infty^2 \leq K^2\sum_\el N_n^2(\el) =K^2 V_n.$$
With $\varepsilon = (\lambda - \sqrt 2) \sqrt{V_{n}}$, (\ref{Hoeffding0}) implies:
\begin{eqnarray*}
&&\PP\big ( \big|\sum_\ell N_{n}(\ell) \, U_\el \big| \geq (\lambda - \sqrt 2) \sqrt{V_{n}} \big ) \\
&&\leq 2 \exp\big (-\frac12 (\lambda - \sqrt 2)^2 \frac{ V_{n} } {K^2 V_{n}} \big) = 2 \exp\big (- {1 \over 2 K^2} (\lambda - \sqrt 2)^2 \big).
\end{eqnarray*}
Let $c, \delta$ be such that $c > \delta > K^2$. In the previous inequality, we take 
$$\lambda= \lambda_n = (2c \log\log n)^\frac12.$$
Let $k(c,\delta)$ be such that $\lambda_{n_k}>\sqrt 2$ and $c(1-\frac2{\sqrt{c\log \log n_k}}) \ge \delta >1$, for $k \geq k(c,\delta)$.
We have:
\begin{eqnarray*}
&&\sum_{k=k(c,\delta)}^\infty \PP\big (\big |\sum_{i=1}^{n_k-1}U_{z_i} \big| \geq \, (\lambda_{n_k} - \sqrt 2) \,\|\sum_{i=1}^{n_k-1} U_{z_i} \|_2 \big) 
\leq 2 \sum_{k=k(c,\delta)}^\infty  \exp\big (- {1 \over 2 K^2} {(\lambda_{n_k} - \sqrt{2})^2}\big) \\
&&\leq \frac{2}{\exp K^2}\sum_{k=k(c,\delta)}^\infty  \exp \big (-{c \over K^2}\log \log n_k) (1-\frac2{\sqrt{c\log\log n_k}}) \big) \\
&&\leq\frac{2}{{\exp K^2}}\sum_{k=k(c,\delta)}^\infty \frac{1}{(k\log\rho)^{\delta \over K^2}}<\infty.
\end{eqnarray*}
Now we can apply (\ref{limsup0}). It follows: 
$$\limsup_n \frac{|\sum_{i=0}^{n-1} U_{z_i}|}{\sqrt{2c (\log\log n) V_n}}\leq 1, \PP\text{-a.e.}$$
Taking $c = c_n$ with $c_n \downarrow K^2$, we get (\ref{lil0}).

\vskip 3mm
\goodbreak
{\it Case 2)} 

For general associated r.v.'s, we use simply that 
$\PP\big (\big |\sum_{i=0}^{n-1}U_{z_i} \big| \geq \, \lambda \,\|\sum_{i=0}^{n-1} U_{z_i} \|_2 \big) \leq \frac1{\lambda^2}$.

We take $\lambda_n = (\log \sigma_n)^{\frac12+\varepsilon}$, with $\varepsilon > 0$. By (\ref{rhok}) we have $\lambda_{n_k} \geq (k \log \rho)^{\frac12+\varepsilon}$, 
and therefore, for a constant $C_1$: $\sum_k \frac1{\lambda_{n_k}^2} \leq C_1 \sum_k k^{-(1+2\varepsilon)} < +\infty$; hence condition (\ref{condSer}). 

Moreover we have $k \log \rho \leq \log \sigma_{n_k} \leq \log \sigma_{n_{k+1}} \leq (k+2) \log \rho$; hence
$${\lambda_{n_{k+1}} \over \lambda_{n_k}} = \bigl({\log \sigma_{n_{k+1}} \over \log \sigma_{n_k)}}\bigr)^{\frac12 + \varepsilon}
\leq (1 + \frac2{k})^{\frac12 + \varepsilon} \to 1.$$
By (\ref{limsup0}), this proves (\ref{casea}) in 2a)

For case 2b) we have 
$\sigma^2_n \leq  V_n \sum_\p \langle U_\p , U_{\0}\rangle$ and $\sigma_n \leq n$, 
hence it yields (\ref{lil00}). The last conclusion in case 2b) is now clear. Remark that it follows also from Proposition \ref{LLN1}. \eop
\end{proof}

\subsection{A Glivenko-Cantelli type theorem}

\

{\bf Empirical process}

Let us consider a random field of r.v.'s $(X_\el, \el \in \Z^d)$ on $(\Omega, \mathcal F,\PP)$ with common distribution function $F$.
Let $(z_k)\subset\mathbb Z^d$ be a sequence with self-intersections $(V_n)$.

{\it Notation.} We say that $(X_\el, \el \in \Z^d)$ satisfies a Glivenko-Cantelli theorem along a sequence $(z_k)$ in $\Z^d$ if
$$\lim_n \sup_s |\frac1n \sum_{k=1}^n {\bf 1}_{(-\infty,s]} (X_{z_k}(\omega)) - F(s)| = 0, \text{ for } \PP\text{-a.e.} \omega.$$
We show now a Glivenko-Cantelli theorem along a sequence $(z_k)$ under various hypotheses on $(z_k)$ and on $(X_\el)$ (mixing,  i.i.d., associated or PQD). 

Le $(X_\el, \el \in \Z^d)$ be a r.f. Denoting by $\sigma(X_\el)$ the $\sigma$-algebra generated by the random variable $X_\el$, 
we define a coefficient of mixing by 
\begin{eqnarray}
\gamma(\el) := \sup_{r \in \Z^d}\sup_{A \in \sigma(X_{\r}), \, B \in \sigma(X_{\el + \r})} |\PP(A \cap B) - \PP(A) \PP(B)|. \label{coefMix0}
\end{eqnarray}
Observe that for every $s, t \in \R$ it holds: 
\begin{eqnarray}
\sup_{\r \in \Z^d} |\langle 1_{X_\r \leq s} - \PP(X_\r \leq s), 1_{X_{\el+\r} \leq t} - \PP(X_{\el+\r} \leq t) \rangle| \leq \gamma(\el), \forall \el \in \Z^d. \label{coefMixEmp}
\end{eqnarray}
By (\ref{coefMixEmp}) and Proposition \ref{LLN1}, we get:
\begin{thm} \label{ratethm} Let $(z_k)$ be such that $\displaystyle V_n \leq C_1 {n^2 \over (\log n)^\beta}$, for constants $C_1 > 0, \beta$.
\hfill \break If $\sum_{\el \in \Z^d} \gamma(\el) < +\infty$ and $\beta > 1$,
\hfill \break or if the r.f. is stationary and $\sum_{\el \in \Z^d} \gamma(\el)^\zeta < +\infty$, for some $\zeta \in [1, 2]$ and $\beta > \zeta$, 
\hfill \break then $(X_\el, \el \in \Z^d)$ satisfies a Glivenko-Cantelli theorem along $(z_k)$.
\end{thm}

Using Proposition \ref{LIL1}, we consider now the  i.i.d. and associated cases.
\begin{thm} \label{indicat} 
a) If $(X_\el)_{\el \in \Z^d}$ is a r.f. of i.i.d. r.v.'s, then under the condition $V_n = o(n^2 \, (\log\log n)^{-1})$ 
it satisfies a Glivenko-Cantelli theorem along $(z_k)$.

b) If $(X_\el)_{\el\in\mathbb Z^d}$ is a strictly stationary r.f. of associated r.v.'s such that $\sum_{\el} \langle X_\el, X_\0\rangle$ converges, then,
under the condition $V_n=O(n^2 \log^{-(1+\eta)}n)$ for some $\eta>0$, for a.e. $\omega$, we have for each continuity point $s$ of $F$:
\begin{eqnarray}
\lim_n \frac1n \sum_{k=0}^{n-1} {\bf 1}_{(-\infty,s]} (X_{z_k}(\omega))=F(s). \label{limGC1}
\end{eqnarray}
If $F$ is continuous, the convergence is uniform in $s$.
\end{thm}
\begin{proof}
a) Denote by $F_n(s)(\omega)$ the means $\frac1n \sum_{k=0}^{n-1} {\bf 1}_{(-\infty,s]} (X_{z_k}(\omega))$.
Let $Q$ be a dense countable set of continuity points of $F$.
 
For every $s\in Q$, by the assumption on $V_n$ and Proposition \ref{LIL1}, there is a null set $N(s)$ such that, for a sequence $\varepsilon_n$ tending to 0,
for every $\omega\not\in N(s)$,
$$|F_n(s)(\omega) - F(s)| \leq \varepsilon_n (V_n \log \log n)^{-\frac12} |\sum_{k=0}^{n-1} {\bf \bigl(1}_{(-\infty,s]}(X_{z_k}) - F(s)\bigr)| \to 0.$$
Then $F_n(s)(\omega) \to F(s)$ for every $\omega$ outside the null set $N:=\cup_{s\in Q} N(s)$ and for $s\in Q$.

Similarly by Proposition \ref{LIL1}, for every $s$ in the set $J$ of jumps of $F$, we have $F_n(s)(\omega)-F_n(s^-)(\omega)\to F(s)-F(s^-)$ a.e.
As $J$ is countable, this convergence holds for every $s\in J$ and $\omega\not \in \tilde N$, where $\tilde N$ is a null set.

Outside the null set $N\cup\tilde N$, Lemma \ref{ChungLem} applied with $Q$ and $J$ implies the result.

b) We consider now the case of a strictly stationary random field of associated r.v.'s. Let $s$ be a continuity point of the common distribution $F$. 
For every $\epsilon>0$ there exists  $\delta>0$, such that $F(s+\delta)-F(s-\delta)\le\epsilon$.
As in \cite{Yu93z}, for $\delta>0$ and $s$, define the approximated step function $h_{\delta,s}$ by
$h_{\delta,s} (x)=0$, if $x\le s-\delta$ and $h_{\delta,s} (x)=1+\frac{x-s}{\delta}$ if $s-\delta\le x\le s$, otherwise, 
$h_{\delta,s} (x)=1$. It is a non decreasing continuous function with $h'_{\delta, s}(x)=1/\delta$ for $s-\delta<x<s$. It follows from the above Hoeffding's identity 
applied to this approximated step function (see \cite{Am17}):
\begin{eqnarray*}
{\text Cov}  (h_{\delta,s}(X_\el),h_{\delta,s}(X_\0))\le \delta^{-2} \langle X_\el, X_\0\rangle,\\
{\text Cov}  (h_{\delta,s+\delta}(X_\el),h_{\delta,s+\delta}(X_\0))\le \delta^{-2} \langle X_\el, X_\0\rangle.
\end{eqnarray*}
By association and non decreasing, $\big ( h_{\delta,s}(X_\el)\big)$ as well as $\big ( h_{\delta,s+\delta}(X_\el)\big)$ are stationary r.f.s of associated r.v.'s,
and we may apply Proposition \ref{LIL1} to their centered versions (also associated).
The condition simply reads, for $\tau=s, s+\delta$:
$$\sum_{\el}{\text Cov}  (h_{\delta,\tau}(X_\el),h_{\delta,\tau}(X_\0))\le \delta^{-2}\sum_\el \langle X_\el, X_\0\rangle<\infty.$$
We put $\overline S_n=\sum_{k=0}^{n-1} h_{\delta,s}(X_{z_k})$
and $\underline S_n =\sum_{k=0}^{n-1} h_{\delta,s+\delta}(X_{z_k})$. By $h_{\delta, s+\delta}(x) \le {\bf 1}_{\{x>s\}}\le h_{\delta, s}(x)$,
it holds $\underline S_n\le \sum_{k=0}^{n-1} {\bf 1}_{(s,\infty)} (X_{z_k}) \le \overline S_n$.
Hence by Proposition \ref{LIL1}, we have almost everywhere
$\frac1n\underline S_n\to \E[h_{\delta, s+\delta}(X_\0) ]\ \text{and}\ \frac1n\overline S_n\to \E[h_{\delta, s}(X_\0)]$.
Since
\begin{eqnarray*}
&\E[h_{\delta, s}(X_\0)]\le F(s)-F(s-\delta)+1-F(s)\le \epsilon+ 1-F(s),\\
&\E[h_{\delta, s+\delta}(X_\0) ]\ge 1-F(s+\delta)=1-F(s)-(F(s+\delta)-F(s))\ge 1-F(s)-\epsilon,
\end{eqnarray*}
we conclude 
$$1-F(s)-\epsilon\le \liminf_n \frac1n \sum_{k=0}^{n-1} {\bf 1}_{(s,\infty)} (X_{z_k})
\le \limsup_n\frac1n \sum_{k=0}^{n-1} {\bf 1}_{(s,\infty)} (X_{z_k})\le 1-F(s)+\epsilon.$$
Subtracting the $1$'s and taking $\epsilon\to 0$, we get (\ref{limGC1}). \eop
\end{proof}

{\it PQD variables.}

The result shown for associated variables can be extended to the class of PDQ variables, but with a stronger condition on the local times of the sequence $(z_k)$.
\begin{proposition} \label{PDQthm} Let $(U_\el)$ be a centered stationary random field of pairwise PQD r.v.'s such that 
$\sum_\el \langle U_{\el}, U_\0 \rangle$ converges. Let $(z_k)$ be a sequence of points
with maximal local times sequence $(M_n)$. If $\sum_{n\geq 1}{M_n \over n^2} < +\infty$,
then $\frac1n (U_{z_0} +\cdots+ U_{z_{n-1}})$ converges a.e. to $0$.
\end{proposition}
\proof We apply the following result of \cite{Bir88}: let $(Y_j: j \ge1)$ be a sequence of pairwise centered PQD r.v.'s. 
with finite variance. If $\sum_{j\ge1}j^{-2}\text{Cov}(Y_j, \sum_{i=1}^j Y_i)$ converges  and $\sup_j \E| Y_j| <\infty$, then $n^{-1} \sum_{i=1}^n Y_i \to 0$ \text{a.e.}

Taking for $Y_j$ the (still) pairwise PQD r.v.'s $U_{z_j}$, we get  the result,
since $\text{Cov}( U_{z_j}, U_{z_1}+\cdots+U_{z_j}) \leq M_j\sum_\el  \langle U_{\0}, U_\el\rangle$.
\eop

Now, we consider the empirical distribution. 
\begin{thm} \label{PDQemp} Let $(X_\el)$ be a centered strictly stationary random field of pairwise PQD r.v.'s with distribution function $F$ such that
$\sum_\el  \langle X_{\el}, X_\0 \rangle$ converges. Let $(z_k)$ be a sequence of points
with maximal local times sequence $(M_n)$. If $\sum_{n \geq1}{M_n \over n^2} < +\infty$,
then for each continuity point $s$ of $F$, we have for a.e. $\omega$:
$\lim_n \frac1n \sum_{k=0}^{n-1} {\bf 1}_{(-\infty,s]} (X_{z_k}(\omega))=F(s)$.

In particular, if $F$ is continuous, the above convergence is uniform over $s$.
\end{thm}
\proof The r.f.s $h_{\delta, s}(X_\el)$ and $h_{\delta, s+\delta}(X_\el)$ are still stationary pairwise PQD. The proof is analogous to the proof of Theorem \ref{indicat}.
For the last statement, we use Lemma \ref{ChungLem}. \eop

{\it Remark.} If $M_n = O(n \, (\log n)^{-(1+\eta)})$, then $V_n=O(n^2 \, (\log n)^{-(1+\eta)})$. 

If $\displaystyle V_n \leq C {n^2 \over (\log n)^\beta}$, with $\beta > 2$, then $\displaystyle \sum_{n \geq1}{M_n \over n^2} < +\infty$.

As shown in Section \ref{exSect}, $\sum_{n\geq1}{M_n \over n^2}$ converges when the sampling is done along random walks, 
but diverges in some examples of sampling along ``deterministic'' random walks. 

\subsection{A sufficient condition for a FCLT for the sampled empirical process} \label{FCLT0sect}

\ 

After a Glivenko-Cantelli theorem for sampled empirical processes, we consider now the Functional Central Limit Theorem (FCLT). 
Let $(z_k)$ be in $\mathbb Z^d, d\ge1,$ with the associated quantities $N_n(\el)$, $M_n$ and $V_n$ defined by (\ref{notaVn}).

Before restricting to a r.f. of  i.i.d. r.v.'s, first we examine the variance in the more general situation where the series of correlations is absolutely summable.

{\it Kernel associated to a sequence $(z_k)$ and variance.}

Let $K_n$ be the kernel (which is a real even function on $\T^d$ depending on $n \geq 0$) defined by the equivalent formulas:
\begin{eqnarray}
K_n(\t) &=& |\sum_{k=0}^{n-1} e^{2\pi i \langle z_k, \t \rangle} |^2 = n +  2 \sum_{k = 1}^{n-1} \, \sum_{j = 0}^{n-k-1} \cos(2\pi \langle z_{k+j} - z_j, t \rangle) 
= |\sum_{\el \in \Z^d} N_n(\el) \, e^{2\pi i \langle \el, t \rangle}|^2
\nonumber \\
&=& n + 2 \sum_\el \bigl(\sum_{k = 1}^{n-1} \, \sum_{j = 0}^{n-k-1} 1_{z_{k+j} - z_j = \el}\bigr) \cos(2\pi \langle \el, t \rangle). \label{Vnp}
\end{eqnarray}
If $(X_\el, \el \in \Z^d)$ is a stationary r.f. such that $\sum_{\el \in \Z^d} |\langle X_\el, X_\0 \rangle| < +\infty$, the spectral density $\rho$ is continuous and we have:
$$\int |\sum_{k=0}^{n-1} X_{z_k}|^2 d\PP = \int_{\T^d} K_n(\t) \, \rho(\t) \, d\t \leq \|\rho\|_\infty V_n \leq (\sum_{\el \in \Z^d} |\langle X_\el, X_\0 \rangle|) V_n.$$
One can ask if there is an asymptotic variance, i.e., a limit for the normalised quantity 
$\displaystyle V_n^{-1} \int |\sum_{k=0}^{n-1} X_{z_k}|^2 d\PP$ which is bounded if the series of correlations is absolutely summable. 

The existence of asymptotic variance is shown in \cite{CohCo17} in the case of summation along a random walk. 
We will come back to the question of its positivity for transient random walks in Subsection \ref{sectRW}.
 
\vskip 3mm
{\bf Functional Central limit Theorem in the i.i.d. case}

The following result gives a sufficient condition for a Functional Central limit Theorem (FCLT) along a sequence $(z_k)$ in the i.i.d. case.

The standard Brownian bridge process $W^0(s)$ is the centered Gaussian process $W^0(s) := W(s) - s W(1)$ in $C(0,1)$, where  $W(s)$ is the Wiener process.
It has the properties $W^0(0)=W^0(1)=0$ and  $\E[W^0(s_1) W^0(s_2)]=s_1\wedge s_2 - s_1 s_2$. 

Let $(X_\k)_{\k\in\mathbb Z^d}$ be i.i.d. random variables with common probability distribution $F$ in $[0,1]$. We put $W_F^0 = W^0\circ F$.
Let $Y_n(s)$ be the random element in $D[0, 1]$ defined by
\begin{eqnarray*}
Y_n(s) = \frac1{\sqrt{V_n}}\sum_{k=0}^{n-1} \, [{\bf 1}_{X_{z_k} \leq s}- F(s)] 
= \frac1{\sqrt{V_n}}\sum_{\el \in \Z^d} N_n(\ell) \, [{\bf 1}_{X_{\el} \leq s} - F(s)].
\end{eqnarray*}
\begin{thm} \label{empirThm1} $Y_n(s) \to_{D[0, 1]} W_F^0(s)$, if $(z_k)$ satisfies the condition
\begin{eqnarray}
\lim_n {M^2_n \over V_n} = 0, \label{condi00}
\end{eqnarray}
\end{thm} \proof \ The result follows from the Cram\'er-Wold device, if we prove convergence of the finite dimensional distributions and tightness. 
The variance is
\begin{eqnarray}
&&\E  [Y_n(s)]^2 = \frac1{{V_n}} \, \sum_{\el} \, N_n^2(\ell) \, \E [{\bf 1}_{X_\el \leq s} - F(s)]^2 = \sigma^2(s) = F(s) (1 - F(s)).
\end{eqnarray}
1) {\it Finite dimensional distributions.} The convergence follows from Lindeberg's theorem for triangular arrays of independent random variables 
as in \cite[thm 7.2]{Bil68}. 
The Lindeberg's condition for the triangular array of independent r.v.'s
$\displaystyle \big(\frac{N_n(\ell)[ {\bf 1}_{X_\el \leq s}- F(s)]}{\sqrt{V_n}}\big)_{\el, n}$ 
follows from
\begin{eqnarray*}
&&\frac1{V_n}\sum_\el \int _{\{N_n(\el)| {\bf 1}_{X_\el \leq s} -  F(s)|
\geq \, \varepsilon \sqrt{V_n}\}} N_n^2(\el) \, |{\bf 1}_{X_\el \leq s}- F(s)|^2d\PP \\
&\leq& \frac1{V_n}\sum_\el  N_n^2(\el) \int _{\{\sup_\el N_n(\el)| {\bf 1}_{X_\0 \leq s}- F(s)| 
\geq \varepsilon \sqrt{V_n}\}} \, |{\bf 1}_{X_\0 \leq s}- F(s)|^2 d\PP \to 0,
\end{eqnarray*}
for every $\varepsilon > 0$, since $V_n=\sum_\el N_n^2(\el)$ and $\displaystyle {\sup_\el N_n(\el) \over \sqrt{V_n}} \to 0$, by assumption.

For the correlation of the process taken at $s_1$ and $s_2$, it holds by independence:
\begin{eqnarray*}
\E[Y_n(s_1)Y_n(s_2)]&=&\frac{1}{V_n}\sum_{\el_1,\el_2}N_n(\el_1)N_n(\el_2)\E[(  {\bf 1}_{X_{\el_1} \leq s_1} - F(s_1))({\bf 1}_{X_{\el_2} \leq s_2} -F(s_2) )]\\
&=&\frac{1}{V_n}\sum_\ell N^2_n(\el)( F(s_1\wedge s_2) - F(s_1) F(s_2))=F(s_1\wedge s_2)-F(s_1)F(s_2).
\end{eqnarray*}
This proves the convergence in distribution: $Y_n(s) \to W_F^0(s)$ for every $s$.

Let us show now the convergence of the finite dimensional distributions. Starting with the asymptotic distribution of $aY_n(s_1)+bY_n(s_2)$, 
by the above computation, we have
\begin{eqnarray}
&&\E[(aY_n(s_1)+bY_n(s_2))^2]= \nonumber\\
&&a^2F(s_1)(1-F(s_1))+b^2 F(s_2)(1-F(s_2))+2ab( F(s_1\wedge s_2) - F(s_1)F(s_2)). \label{asympVar0}
\end{eqnarray}
As above, it is easily seen that Lindeberg's condition is satisfied for the triangular array defined by $aY_n(s_1) + bY_n(s_2)$.
It means that the asymptotic distribution of $aY_n(s_1) + bY_n(s_2)$ is centered Gaussian with variance as computed above.

Note that $\E[(aW^0(s_1) + bW^0(s_2))^2]$ is also given by (\ref{asympVar0}) above.

Similarly, for every $s_1 \leq \cdots \leq s_r$, it holds
$$(Y_n(s_1),\cdots, Y_n(s_r))  \to_{dist} (W_F^0(s_1),\ldots, W_F^0(s_r)).$$

{\it Tightness.}  First we suppose $F$ continuous. Following the method of \cite{Bil68}, it is enough to show that for $s\leq t \leq r$, uniformly in $n$,
\begin{eqnarray*}
\E[(Y_n(t)-Y_n(s))^2 (Y_n(r)-Y_n(t))^2]\le C(F(r)-F(s))^2.
\end{eqnarray*} 
Putting $F(u,v) := F(v)-F(u)$, $f(\el, u, v) := {\bf 1}_{u< X_\el \leq v} - F(u,v)$, for $u < v$, we compute
\begin{eqnarray*}
\E[(Y_n(t)-Y_n(s))^2 (Y_n(r)-Y_n(t))^2] = 
\frac1{V^2_n} \E\big[\big(\sum_{\el } N_n(\el) f(\el,s,t) \big)^2 \big(\sum_{\el} \, N_n(\el) f(\el,t,r) \big)^2 \big].
\end{eqnarray*}
By expansion and independence, the above expression is sum of three types of terms:
$$\frac1{V^2_n}\sum_\el N_n^4(\el) \, [A], \, \frac1{V^2_n}\sum_{\el_1, \el_2}N_n^2(\el_1) N_n^2(\el_2) \, [B], 
\, \frac1{V^2_n}\sum_{\el_1\not=\el_2}N_n^2(\el_1) N_n^2(\el_2)  \, [C],$$
\begin{flalign*}
\text{ with } &A=\E[f^2(\el, s,t) f^2(\el,t,r)], B=\E[f^2(\el_1, s,t)] \, \E[f^2(\el_2,t,r)],&\\ 
&C=\E[ f(\el_1, s,t) f(\el_1, t,r))] \, \E[ f(\el_2, s,t) f(\el_2, t,r))].&
\end{flalign*}
By stationarity and since the intervals do not overlap, we have
\begin{eqnarray*}
&A&=F(s,t)F^2(t,r)+F^2(s,t)F(t,r)-3F^2(s,t)F^2(t,r), \\
&B&=F(s,t)(1-F(s,t))\cdot F(t,r)(1-F(t,r)), \, C=F^2(s,t)F^2(t,r).
\end{eqnarray*}
Since $0\le F(s,t), F(t,r), F(s,r)\le 1$ and $F(s,t), F(t,r) \le F(s,r)$, it follows
\begin{eqnarray*}
A\le 2F^3(s,r)\le 2F^2(s,r),\, B\le F^2(s,r),\ C\le F^4(s,r) \leq F^2(s,r).
\end{eqnarray*}
Recall that $V_n= \sum_\el N_n^2(\el)$. Using $\|\cdot\|_{\ell_4} \leq \|\cdot \|_{\ell_2}$ for the bound of the first term, we have for some fixed constant $C>0$:
\begin{eqnarray*}
\E[(Y_n(t)-Y_n(s))^2 (Y_n(r)-Y_n(t))^2] \leq C(F(r)-F(s))^2,\ \forall n.
\end{eqnarray*}
Hence by \cite[Theorem 15.6]{Bil68}, for non decreasing continuous $F$, the sequence of processes $(Y_n(s))$ is tight in $D(0,1)$.
This proves that, if $F$ is continuous, then $Y_n\to_{D(0,1)} (W^0\circ F)$.

Now, for a general $F$ a classical method is to use a generalized inverse. Let us describe it briefly. We consider first the uniform empirical process.
Let $(\zeta_k)$ be uniformly distributed i.i.d. r.v.'s. Denote the empirical process along $(z_k)$ with respect to $(\zeta_k)$ by $U_n(s)$.
By applying what we have just proved for a continuous distribution, $U_n(s)\to_{D(0,1)}W^0(s)$.

Now let $F^{-1}(t):=\inf \{s: t \le F(s) \}$. We get $\PP(F^{-1}(\zeta_0)\le s)= \PP(X_0\leq s)=F(s)$,
so $Y_n(s)=_{dist.}U_n(F(s))$. We may then proceed as in Billingsley (\cite[Theorem 5.1]{Bil68}) to deduce the FCLT for $Y_n(s)$ with $W^0(F(s))$ as limit. \eop

\section{\bf Local times for ergodic sums} \label{genCocy}

In the previous section about limit theorems for the empirical process sampled along $(z_k)$, 
we have found sufficient conditions on the quantities $V_n$ and $M_n$ associated to $(z_k)$.
When $(z_k)$ is given by a ``cocycle'', $z_k = S_kf(x)$, one can ask if these conditions are satisfied.

We start with some general facts and construct counterexamples for which condition (\ref{condi00}) is not satisfied.
In the next section, we will discuss two very different examples of cocycles:  first the case of random walks, then ``stationary random walks'' generated by a rotation. 

\subsection{Auxiliary general results} \label{genCocy1}

\

First we introduce some notation and make general remarks.
\begin{nota}\label{notaVntf}{\rm Let $(X, {\cal B}, \mu)$  be a probability space and $T$ a measure preserving transformation acting on $X$
such that the dynamical system $(X, {\cal B}, \mu, T)$ is ergodic.

Let $f$ be a measurable function on $X$ with values in $\Z^d$, $d \geq 1$. Its ergodic sums generated by the iteration of $T$, denoted by $f_k$ (or $S_k f$), are
$$f_k(x) := \sum_{j=0}^{k-1} f(T^j x), k \geq 1, \ f_0(x) = 0.$$ 
The sequence $(f_k(x), k \geq 1)$ can be viewed as a ``stationary random walk'' defined on $(X, {\cal B}, \mu)$.
It will be called a ``cocycle'' and denoted by $(\mu,T, f)$ or simply $(T, f)$. 

For $x \in X$, we put (cf. (\ref{notaVn})) $N_0(x, \el) = 0$ and, for $n \geq 1$,}
\begin{eqnarray*}
N_n(T, f, x, \el) &:=&\#\{1 \leq k \leq n:\ f_k(x)=\el\}, \, \el \in \Z^d, \\
M_n(T, f, x) &:=& \max_{\el \in \Z^d} N_n(T, f, x, \el),\\
V_n(T, f, x) &:=& \#\{1 \leq j,k \leq n :\, f_j(x) = f_k(x)\} = \sum_{\el \in \Z^d} \, N_n^2(x, \el).
\end{eqnarray*}
\end{nota}
Most of the time, we will omit $T$ and $f$ in the notation and write simply $N_n(x, \el)$, $M_n(x)$, $V_n(x)$. We have
$\sum_\el N_n(x, \el) = n \text{ and } n \leq V_n(x) \leq n \, M_n(x)$.

{\it A question is to know if the following conditions hold for a.e. $x$:}
\begin{eqnarray}
&&V_n(x) = o(n^2 \, (\log\log n)^{-1}) \text{ or } V_n(x) \leq C_1 {n^2 \over (\log n)^\beta}, \text{ with } \beta > 1,\label{GC} \\
&&\lim_n {M^2_n(x) \over V_n(x)} = 0.\label{condi000}
\end{eqnarray}
For a random walk this is related to a question studied in \cite{ET60} and later in \cite{DPRZ01}: 
How many times does the walk revisit the most frequently visited site in the first $n$ steps?

{\it Cylinder map.} We denote by $\tilde T_f$ the map (sometimes called cylinder map) $(x, \el) \to (Tx, \el + f(x))$ acting on $X \times \Z^d$, 
endowed with the infinite invariant measure $\tilde \mu$ defined as the product of $\mu$ by the counting measure on $\Z^d$.

For $\varphi: X \times \Z^d \to \R$ the ergodic sums for the cylinder map are
$$\tilde S_n \varphi(x, \el) := \sum_{k=0}^{n - 1} \varphi(\tilde T_f^k (x, \el)) = \sum_{k=0}^{n - 1} \varphi(T^k x, \el+ f_k(x)).$$
With $\varphi_0 := {\bf 1}_{X \times \{\0\}}$, it holds
$$\tilde S_n \varphi_0(x, -\el)  = \sum_{k=0}^{n-1} {\bf 1}_{X \times \{\0\}} (T^k x, - \el+ f_k(x)) = \#\{0\le k \leq n-1: \, f_k(x)=\el\}.$$ 
Therefore,  $\tilde S_n \varphi_0 (x, -\el)  = N_n(\el)(x)$.

\vskip 3mm
{\it Recurrence/transience.} It can be shown that a cocycle $(\mu,T, f)$ (over an ergodic dynamical system) is either {\it recurrent} or {\it transient}.
For $f$ with values in $\Z^d$, it means that either $S_k f(x) = 0$ infinitely often for a.e. $x$,
or $S_k f(x) = 0$ finitely often  for a.e. $x$. In the latter case, we have $\lim_k |S_k f(x)| = +\infty$, $\mu$-a.e.

Let ${\Cal R}_n(x) = \{\el \in \Z^d: \ f_k(x) = \el \text{ for some } k \leq n\}$  be the ``range'' of the cocycle, i.e., the set of points visited by $f_k(x)$ up to time $n$. 

In \cite{DGK18} the following result is shown (for the general case of a cocycle with values in a countable group): let $G$ be a countable group and $f : X \to G$. 
If the cocycle $(T, f)$ is recurrent, then $\Card({\Cal R}_n(x)) = o(n)$ for a.e. $x$. If it is transient, there exists $c > 0$ such that $\Card({\Cal R}_n(x)) \sim c \, n$ for a.e. $x$.

Using the lemma below, this implies for a.e. $x$:
\begin{eqnarray}
&&\liminf_n {V_n(x) \over n} > 0 \text{ in the transient case}, \ {V_n(x) \over n} \to + \infty \text{ in the recurrent case}. \label{dico0}
\end{eqnarray}
To show (\ref{dico0}) we use the following inequality valid for a general sequence $(z_k)$:
\begin{lem} If $\Cal A$ is a non empty subset in $\Z^d$, we have:
\begin{eqnarray}
V_n\geq {\bigl(\sum_{k=0}^{n-1} 1_{z_k \,  \in \Cal A}\bigr)^2  \over  \Card(\Cal A)}. \label{majAn}
\end{eqnarray}
\end{lem}
\proof \ Cauchy-Schwarz inequality implies: 
\begin{eqnarray*}
&&\sum_{k=0}^{n-1} 1_{z_k \,  \in \Cal A} = \sum_{\el \in \Cal A} \sum_{k=0}^{n-1} 1_{z_k \,  = \el} 
\leq (\sum_{\el \in A} (\sum_{k=0}^{n-1} 1_{z_k \,  = \el})^2)^\frac12 \, (\Card(\Cal A))^\frac12 \leq V_n^\frac12 \,  \, (\Card(\Cal A))^\frac12. \ \eop
\end{eqnarray*}
If $z_k = S_k f(x)$, this show (\ref{dico0}). Indeed by taking $\Cal A = \Cal R_n(x)$ we get
\begin{eqnarray}
V_n(x) \geq {n^2 \over \Card(\Cal R_n(x))}. \label{growth0}
\end{eqnarray}

\vskip 3mm
\begin{lem} The following formulas hold for quantities defined in \ref{notaVntf}.
\begin{eqnarray}
V_{n}(x) &=& n + 2 \sum_{k = 1}^{n-1} \, \sum_{j = 0}^{n-k-1} (1_{f_k (T^j x) = \0}), \label{expVnx0} \\
&=& 2[N_{n-1}(Tx, 0) + N_{n-2}(T^2 x, 0) + ... + N_1(T^{n-1} x, 0) ] + n, \ n \geq 2, \label{expVnx} \\
M_n(x) &=& \max[N_{n}(x, 0), 1 + \max_{1 \leq k \leq n-1} N_{n-k}(T^k x, 0)] \leq 1 + \max_{0 \leq k \leq n-1} N_{n}(T^k x, 0), \label{maxform} \\
&=& M_{n-1}(Tx) + 1_{\el(n-1, Tx) = \0} \leq M_{n-1}(Tx) + 1. \label{Mn1}
\end{eqnarray}
\end{lem}
\proof \ 
a) From $f_k(x) = f(x) + f_{k-1}(Tx), \, k \geq 1$, it follows 
\begin{eqnarray}
N_n(x, \el) = N_{n-1}(Tx, \el - f(x)) + 1_{f(x) = \el}, \, n \geq 1. \label{Nn1}
\end{eqnarray}
Therefore we have:
\begin{eqnarray*}
&&\sum_{\el \in \Z^d} N_n^2(x, \el) = \sum_{\el \in \Z^d} [N_{n-1}(Tx, \el - f(x)) + 1_{f(x) = \el}]^2 
=  \sum_{\el \in \Z^d} [N_{n-1}(Tx, \el) + 1_{\el = \0}]^2 \\
&&= \sum_{\el \not = 0} [N_{n-1}(Tx, \el)]^2 + [N_{n-1}(Tx, \0) + 1]^2 = \sum_{\el} [N_{n-1}(Tx, \el)]^2 +  2N_{n-1}(Tx, \0) + 1.
\end{eqnarray*}
Hence the relation
\begin{eqnarray}
&&V_n(x) =V_{n-1}(Tx) + 2N_{n-1}(Tx, \0) + 1. \label{vnx}
\end{eqnarray}
We have $V_1(x) = 1$ and by the previous relation we get by induction (\ref{expVnx0}) and (\ref{expVnx}).

b) For $x \in X$, let $\el(n, x)$ (a most visited site by $S_k(x)$ up to time $n$) be defined by
\begin{eqnarray*}
\el(n, x) &:=& \0, \text{ if } N_n(x, \0) \geq N_n(x, \el), \text{ for all } \el \not = 0, \\
\text{ else } &:=& \el_1, \text{ if } \ell_1 \text{ is such that } M_n(x) = N_n(x, \el_1) > N_n(x, \0).
\end{eqnarray*}
Let $p_n(x) \in [1, n]$ be the first visit of $S_k(x)$ to $\el(n, x)$ for $k = 1, ..., n$. By definition $M_n(x) = N_n(x, \el(n, x))$. 

We have $M_n(x) = N_{n}(x, 0)$ if $\el(n, x) = 0$, else
$M_n(x)=N_{n - p_n(x)}(T^{p_n(x)} x, 0) + 1$, by the cocycle relation $S_{p_n(x) + k}(x) - S_{p_n(x)}(x)  = S_k(T^{p_n(x)} x)$. This implies: 
$$M_n(x) \leq \max [N_{n}(x, 0), N_{n - p_n(x)}(T^{p_n(x)} x, 0) +1] \leq \max [N_{n}(x, 0), \max_{1 \leq k \leq n} N_{n - k}(T^k x, 0) +1].$$

It follows (noticing that $N_0(x, 0) =0$):
\begin{eqnarray}
M_n(x) \leq 1 + \max_{0 \leq k \leq n-1} N_{n-k}(T^k x, 0) \leq 1 + \max_{0 \leq k \leq n-1} N_{n}(T^k x, 0) . \label{maxform0}
\end{eqnarray}
This shows (\ref{maxform}).

c) Observe also the following relation: by (\ref{Nn1}) we have:
\begin{eqnarray*}
M_n(x) &=& \sup_\el [N_{n-1}(Tx, \el - f(x)) + 1_{\el - f(x) = \0}] = \sup_\el [N_{n-1}(Tx, \el) + 1_{\el= \0}] \\
&=& \max \, [\sup_{\el \not = \0} N_{n-1}(Tx, \el), \, N_{n-1}(Tx, \0) + 1].
\end{eqnarray*}
If $\el(n-1, Tx) = \0$, then $N_{n-1}(Tx, \0) \geq \sup_{\el \not = \0} N_{n-1}(Tx, \el)$.
If $\el(n-1, Tx) \not = \0$, then $N_{n-1}(Tx, \0) < \sup_{\el \not = \0} N_{n-1}(Tx, \el)$. 
This shows (\ref{Mn1}).
\begin{rem} \label{unifx}
By (\ref{maxform}), if $K_n$ is a uniform bound over $x$ of $N_n(x, \0)$, then $M_n(x) \leq K_n$. 
Likewise, if $N_n(x, \0) \leq K_n$, for a.e. $x$, then $M_n(x) \leq K_n$, for a.e. $x$.
\end{rem}

Now we show that the set of $x \in X$ such that $\lim_n {M_n^2(x) \over V_n(x)} = 0$ has measure 0 or 1.
\begin{lem} \label{MnInv} It holds: $\displaystyle \lim_n \, [{M_n^2(x) \over V_n(x)} -  {M_{n}^2(Tx) \over V_{n}(Tx)}] = 0$.
\hfill \break If $T$ is ergodic, there is a constant $\gamma \in [0, 1]$ such that $\displaystyle \limsup_n {M_n^2(x) \over V_n(x)}= \gamma$ for a.e. $x$.
\end{lem}
\proof \ We use (\ref{vnx}) and (\ref{Mn1}). Putting $\varepsilon =  1_{\el(n-1, Tx) = \0}$, we have:
\begin{eqnarray*}
&&|{M_n^2(x) \over V_n(x)} - {M_{n-1}^2(Tx) \over V_{n-1}(Tx)}| = |{M_{n-1}^2(Tx) 
+ \varepsilon (2 M_{n-1}(Tx) + 1) \over V_{n-1}(Tx) +  2N_{n-1}(Tx, \0) + 1} -  {M_{n-1}^2(Tx) \over V_{n-1}(Tx)}| \nonumber\\
&&= |{\varepsilon (2 M_{n-1}(Tx) + 1) \over V_{n}(x)} - {(2N_{n-1}(Tx, \0) + 1) \over V_{n}(x)} \, {M_{n-1}^2(Tx) \over V_{n-1}(Tx)}|\nonumber \\
&& \leq {2 M_{n-1}(Tx) + 1 \over V_{n}(x)} + {2N_{n-1}(Tx, \0) + 1 \over V_{n}(x)} \leq {4M_{n-1}(Tx) \over V_{n}(x)} + {2 \over V_{n}(x)} 
\leq {4 \over \sqrt n} + {2 \over V_{n}(x)}.
\end{eqnarray*}
For the last inequality we use that either $M_n(x) \geq \sqrt n$, hence ${M_n(x) \over V_n(x)} \leq {1 \over M_n(x)} \leq {1 \over \sqrt n}$, or
$M_n(x) < \sqrt n$, hence ${M_n(x) \over V_n(x)}\leq {\sqrt n \over n} = {1 \over \sqrt n}$. Therefore, 
\begin{eqnarray}
&&|{M_n^2(x) \over V_n(x)} -  {M_{n-1}^2(Tx) \over V_{n-1}(Tx)}| \to 0. \label{MnMTn1}
\end{eqnarray}
Observe now that
\begin{eqnarray*}
M_n(x) = M_{n-1}(x) + \varepsilon_n, \text{ where } \varepsilon_n = 0 \text{ or } =1,
\end{eqnarray*}
and $\varepsilon_n =1$ if and only if there is $\el_n$ such that 
$$M_{n-1}(x) = \#\{1 \leq k \leq n - 1:\ f_k(x)=\el_n\} \text{ and } f_n(x)=\el_n.$$
We have
$$M_n^2(x) =  M_{n-1}^2(x) + c_n, \text{ with } c_n = \varepsilon_n (1 + 2  M_{n-1}(x))$$
and $N_n(x, \el) = N_{n-1}(x, \el) + \varepsilon'_n(\el)$, with $\varepsilon'_n(\el) = 1_{f_n(x)=\el}$ and $\sum_{\el \in \Z^d} \, \varepsilon'_n(\el) = 1$.

Therefore,
\begin{eqnarray*}
&&V_n(x) = \sum_{\el \in \Z^d} \, (N_{n-1}(x, \el) + \varepsilon'_n(\el))^2 = V_{n-1}(x)  + 2 \sum_{\el \in \Z^d} \, \varepsilon'_n(\el) N_n(x, \el)) + 1, \\
&&0 \leq V_{n}(x)  =  V_{n-1}(x)  + d_n, \text{ with } d_n \leq 2 M_n(x) + 1.
\end{eqnarray*}
\begin{eqnarray*}
&&|{M_n^2(x) \over V_n(x)} -  {M_{n-1}^2(x) \over V_{n-1}(x)}| = |{M_{n-1}^2(x) + c_n \over V_{n-1}(x) + d_n} -  {M_{n-1}^2(x) \over V_{n-1}(x)}| 
\leq {c_n \over V_{n}(x)} + {d_n \over V_n(x)} \, {M_{n-1}^2(x) \over V_{n-1}(x)}\\
&&\leq {\varepsilon_n (1 + 2  M_{n-1}(x) \over V_{n}(x)} + {2 M_n(x) + 1 \over V_n(x)} \, {M_{n-1}^2(x) \over V_{n-1}(x)}
\leq {2 \over V_{n}(x)} + {4 M_n(x) \over V_n(x)} \leq {2 \over V_{n}(x)} + {4 \over \sqrt n} \to 0. 
\end{eqnarray*}
From (\ref{MnMTn1}) and the convergence above, it follows $\lim_n \, [{M_n^2(x) \over V_n(x)} -  {M_{n}^2(Tx) \over V_{n}(Tx)}] = 0$.
By ergodicity of $T$, this shows the lemma \eop

\vskip 3mm
{\bf Case of a coboundary}

The case when $f$ is coboundary degenerates. Indeed, the following holds:
\begin{proposition} \label{cobThm} 
If $f$ is a coboundary we have:
\hfill \break a) $\liminf_n {M_n(x)  \over n} > 0$, for a.e. $x$;
\hfill \break b) there is  a constant $\beta > 0$ such that ${\frac1{n^2}} V_n(x) \to \beta$, for a.e. $x$;
\hfill \break c) for a.e. $x$, $\liminf_n {M_n^2(x) \over V_n(x)} > 0$.
\end{proposition}
\proof Suppose that $f$ is coboundary, $f = T\Phi - \Phi$. Since $f$ has values in $\Z^d$ and $T$ is ergodic, for all component $\Phi_j$ of $\Phi$,
$e^{2\pi i \Phi_j}$ is a constant. It follows that  $\Phi $ has also its values in $\Z^d$ up to an additive constant and we can assume that $\Phi $ has values in $\Z^d$.

a) We have $\liminf_n {M_n(x)  \over n} \geq \liminf_n \frac1n N_n(x, \0) > 0$, for a.e. $x$. The positivity results the following simple argument:

For $R \geq 1$, let $A_R$ denote the set $\cup_{\el: \|\el\| \leq R} (\Phi = \el)$. 
Since, for each $\el$, by Birkhoff's theorem, $\lim_n \frac1n  \sum_{0\le k \leq n-1} 1_{\Phi(T^k x) = \el} = \mu(\Phi = \el)$, it holds
$$\frac1n N_n(x, \0) \geq \sum_{\el \in A_R} \, 1_{(\Phi = \el)}(x) \, \frac1n \sum_{k=0}^{n-1} 1_{(\Phi=\el)}(T^k x) 
\to \sum_{\el \in A_R} \, 1_{(\Phi = \el)}(x) \, \mu(\Phi= \el).$$
Therefore, for every $R \geq 1$, $\liminf_n {M_n(x)  \over n} \geq \liminf_n {N_n(x, 0)  \over n}  \geq 
\sum_{\el \in A_R} \, 1_{(\Phi = \el)}(x) \, \mu(\Phi= \el)$, and the limit when $R \to \infty$ at right is $>0$, for a.e. $x$.

b) For $V_n$, we have:
\begin{eqnarray*}
V_n(f, x) &&= \sum_{\el \in \Z^d} \, N_n^2(x, \el) = \sum_{\el \in \Z^d} \, \#\{0 \le k \leq n-1: \ \Phi(T^k x) - \Phi(x) = \el\}^2 \\
&&= \sum_{\el \in \Z^d} \, \#\{0 \le k \leq n-1:\ \Phi(T^k x)=  \el\}^2  =  \sum_{\el \in \Z^d} \, \bigl(\sum_{0\le k \leq n-1} 1_{\Phi(T^k x) = \el}\bigr)^2,
\end{eqnarray*}
$$\text{hence: }{\frac1{n^2}}\sum_{\el \in A_R} \, \bigl(\sum_{0\le k \leq n-1} 1_{\Phi(T^k x) = \el}\bigr)^2 
= \sum_{\el \in A_R} \, \bigl({\frac1{n}} \sum_{0\le k \leq n-1} 1_{\Phi(T^k x) = \el}\bigr)^2\to \sum_{\el \in A_R} \, (\mu(\Phi=\el))^2.$$
This implies, for every $R \geq 1$,
$$\liminf_n {\frac1{n^2}} \sum_{\el \in \Z^d} \, \bigl(\sum_{0\le k \leq n-1} 1_{\Phi(T^k x) = \el}\bigr)^2 \geq
\lim_n {\frac1{n^2}}\sum_{\el \in A_R} \, \bigl(\sum_{0\le k \leq n-1} 1_{\Phi(T^k x) = \el}\bigr)^2 = \sum_{\el \in A_R} \, (\mu(\Phi=\el))^2.$$
It follows:
$\displaystyle \liminf_n {\frac1{n^2}} \sum_{\el \in \Z^d} \, \bigl(\sum_{0\le k \leq n-1} 1_{\Phi(T^k x) = \el}\bigr)^2 \geq \sum_{\el \in \Z^d} \, \mu(\Phi=\el)^2$.
For the complementary of $A_R$, it holds:
\begin{flalign*}
&\sum_{\el: \|\el\| > R}  \bigl(\sum_{0 \leq k < n} 1_{\Phi(T^k x) = \el}\bigr)^2 =   
\sum_{0 \leq j, k < n} \, \sum_{\el: \|\el\| > R} 1_{\Phi(T^j x) = \el} \, 1_{\Phi(T^k x) = \el}& \\
&\leq \sum_{0 \leq j, k < n} \, (\sum_{\el: \|\el\| > R} 1_{\Phi(T^j x) = \el}) \, (\sum_{\el: \|\el\| > R} 1_{\Phi(T^k x) = \el})
\leq \sum_{0 \leq j, k < n} 1_{A_R^c(T^j x)} 1_{A_R^c(T^k x)} = \bigl(\sum_{0 \leq k < n} 1_{A_R^c(T^k x})\bigr)^2.&
\end{flalign*}
It follows for the upper bound:
\begin{eqnarray*}
&&\limsup_n {\frac1{n^2}} \sum_{\el \in \Z^d} \, \bigl(\sum_{0\le k \leq n-1} 1_{\Phi(T^k x) = \el}\bigr)^2 \\
&&\leq \lim_n \sum_{\el \in A_R} \, \bigl( {\frac1n} \sum_{0\le k \leq n-1} 1_{\Phi(T^k x) = \el})^2
+ \lim_n \bigl({\frac1n}\sum_{0 \leq k < n} 1_{A_R^c(T^k x)} \bigr)^2 \\
&&= \sum_{\el \in A_R} \, (\mu(\Phi=\el))^2 + \mu(A_R^c)^2 \underset{R \to \infty} \to \sum_{\el \in \Z^d} \, \mu(\Phi=\el)^2.
\end{eqnarray*}
This shows b) with $\beta = \sum_{\el \in \Z^d} \, \mu(\Phi=\el)^2 > 0$.

c) Follows from a) and b). \eop

\begin{proposition} \label{Vno}There is a constant $\beta \geq 0$ such that, for a.e. $x$, $\lim_n {V_{n}(x) \over n^2} = \beta$. 
We have $\beta > 0$ if and only if the cocycle $(T, f)$ is a coboundary.
\end{proposition}
\proof \ The case of a coboundary follows from Proposition \ref{cobThm}. 

Suppose now that the cocycle is not a coboundary. From (\ref{expVnx0}), we can write
\begin{eqnarray*}
{V_{n}(x) \over n^2}&&= \frac1n + \frac2n \sum_{k = 1}^{n-1} \, \frac1n\sum_{j = 0}^{n-k-1} (1_{f_k (T^j x) = \0}) \\
&&\leq \frac1n + \frac2n \sum_{k = 1}^{n-1} \, \frac1n\sum_{j = 0}^{n-1} (1_{f_k (T^j x) = \0}) 
= \frac1n + \frac2n \sum_{j = 1}^{n-1} \, {N_n(T^j x, 0) \over n}.
\end{eqnarray*}
We will show that $\displaystyle \frac1n \sum_{j=0}^{n-1} \, {N_n(T^j x, 0) \over n}$ tend to 0 a.e.

By the ergodic theorem of Dunford and Schwarz (in the space of infinite measure $X \times \Z$) applied to $\tilde T_f$ and $\phi_0 = {\bf 1}_{X \times \{0\}}$, 
which is bounded and in $L^p(X \times \Z)$, for every $p \geq 1$, we get
a function $\tilde \phi_0(x)$ which is $\tilde T_f$-invariant and in $L^1(X \times \Z)$ and
$$\lim_n {N_n(x, 0) \over n} = \tilde \phi_0(x), \text{ a.s.}$$
As $f$ is not a coboundary, $\tilde \phi_0$ is zero a.e. (cf. for instance \cite{Co76}.)

Observe that $\|\sup_{n \geq L} {N_n(x, 0) \over n}\|_2 \to 0$, as $L$ goes to $+\infty$.
Indeed, for every $0 < \varepsilon \leq 1$, letting $A_{\varepsilon, L} :=\{x: \sup_{n \geq L} {N_n(x, 0) \over n} > \varepsilon\}$, 
we have $\mu(A_{\varepsilon, L}) \to 0$, when $L \to +\infty$. Since ${N_n(x, 0) \over n} \leq 1$, it follows, for $L$ big enough:
$$\int \bigl(\sup_{n \geq L} ({N_n(x, 0) \over n})\bigr)^2 \, d\mu \leq \varepsilon^2 + \mu(A_{\varepsilon, L}) \leq 2 \varepsilon.$$

We put $\displaystyle \Lambda_n(x) := \sup_{s \geq n} {N_s(x, 0) \over s}$. By the previous observation, we have $\lim_n \|\Lambda_n\|_2 = 0$.

Let us consider the following maximal function for the action of $T$:
\begin{eqnarray}
\tilde \Lambda_n(x) =  \sup_{1 \leq r < \infty} \frac1r \sum_{j=0}^{r-1} \Lambda_n(T^j x) 
=\sup_{1 \leq r < \infty} \frac1r \sum_{j=0}^{r-1} \sup_{s \geq n} {N_s(T^j x, 0) \over s}. \label{lambd}
\end{eqnarray}
From a classical maximal inequality, we have $\|\tilde \Lambda_n\|_2 \leq 2 \|\Lambda_n\|_2 \to 0$.

Observe also that, from the definition of $\tilde \Lambda_n$ in (\ref{lambd}), the following inequalities hold:
$$\tilde \Lambda_n(x) \geq \sup_{r, s \geq n} \frac1r \sum_{j=0}^{r-1} \, {N_s(T^j x, 0) \over s}
\geq \frac1n \sum_{j=0}^{n-1} \, {N_n(T^j x, 0) \over n}.$$
The sequence $\displaystyle \sup_{r, s \geq n} \frac1r \sum_{j=0}^{r-1} \, {N_s(T^j x, 0) \over s}$ is non negative and decreasing. 
Since $\|\tilde \Lambda_n\|_2 \to 0$, the $L_2$-norm of its limit in $(X, \mu)$ is zero.
The result follows. \eop

\begin{rem} \label{lebSpec} { (see also section \ref{sectGC0} and \cite{LLPVW02})  

Let $(U_\el)_{\el \in \Z^d}$ be a r.f. of square integrable r.v.'s on a probability space $(\Omega, \mathcal F, \PP)$ 
stationary in the weak sense and such that $\sum_\el |\langle U_\el , U_{\0}\rangle| < +\infty$. By (\ref{norm2}) and Proposition \ref{Vno}, 
if $f$ is not a coboundary, it holds 
$$\frac1{n^2} \, \|\sum_{k=0}^{n-1} U_{f_k(x)}\|_2^2 \, \leq C \, {V_n(x) \over n^2} \to 0, \text{ for } \mu\text{-a.e. } x.$$

Another result of norm convergence whose proof is like the proof of Proposition \ref{LLN1} is the following. 
Suppose that the r.f. is stationary. Let $\varphi$ be an observable on the dynamical system $(\Omega, \PP, \theta)$ with a spectral measure $\nu_\varphi$. We have:
$$\int_\Omega |\sum_{j=0}^{n-1} \varphi \circ \theta^{z_j}|^2  \, d \PP = \int_{\T^1} |\sum_{j=0}^{n-1} e^{2 \pi i z_jt}|^2 \, d \nu_\varphi(t).$$

Assume that $\nu_\varphi$ is absolutely continuous with respect to the Lebesgue measure on the torus, and let $\rho \in L^1(d\t)$ such that $d\nu_\varphi(\t) = \rho(\t) d\t$.
For $\varepsilon > 0$ there is $M$ such that $\int_{\rho > M} \, \rho \, d\t  < \varepsilon$. We have
\begin{eqnarray*}
\frac1{n^2} \, \int_{\T^d} |\sum_{j=0}^{n-1} e^{2 \pi i \langle z_j, \t \rangle}|^2 \, d \nu_\varphi(\t) 
&\leq& \frac{M}{n^2} \, \int_{\T^d} |\sum_{j=0}^{n-1} e^{2 \pi i \langle z_j, \t \rangle}|^2 \, d\t+ \int_{\rho > M} \, \rho \, d\t 
\leq M {V_n \over n^2}+ \varepsilon.
\end{eqnarray*}
This shows that $\displaystyle {V_n \over n^2} \to 0$ implies 
$\displaystyle \frac1{n^2} \, \int_\Omega |\sum_{j=0}^{n-1} \varphi \circ \theta^{z_j}|^2 \, d\PP \to 0$.
This is satisfied by every $\varphi \in L^2(\PP)$, if the dynamical system has a Lebesgue spectrum.

In particular, taking $z_k = f_k(x)$, by Proposition \ref{Vno}, if $f$ is not a coboundary, it holds 
$$\frac1{n^2} \, \int_\Omega |\sum_{j=0}^{n-1} \varphi(\theta^{f_j(x)} \omega)|^2 \, d\PP(\omega) \to 0, \text{ for a.e. } x.$$
When the spectral density is square integrable, as we have seen in Proposition \ref{LLN1}, the pointwise convergence holds under quantitative hypothesis on the sequence $(z_k)$.
}\end{rem}

\subsection{Non centered cocycles} \label{noncent}

\

In an ergodic dynamical system $(X, \mu, T)$, if $f: X \to \R$ is an integrable function with $\mu(f) >0$, 
by the ergodic theorem for the ergodic sums $S_n^T f(x) = \sum_{k=0}^{n-1} f(T^k x)$, it holds for a.e. $x$: $\lim_n \frac 1n S_nf(x) > 0$ 
and therefore $\lim_n S_n^T f(x) = + \infty$.
If $f$ has values in $\Z$, as the process $S_n^Tf(x)$ visits finitely often each site, one can think there is a chance that the following condition is satisfied:
\begin{eqnarray}
\lim_n {M^2_n(T, f, x) \over V_n(T, f, x)} = 0.\label{condi00x}
\end{eqnarray}

A case where (\ref{condi00x}) is satisfied is the following: let $X$ be a topological compact space, $T: X \to X$ a continuous map, which is uniquely ergodic 
with $\mu$ as unique invariant measure. Let $f: X \to \Z$ be an integrable function such that $\mu(f) \not = 0$. 
Assume $f$ to be Riemann-integrable (i.e. such that, for every $\varepsilon > 0$, there are two continuous functions 
$\psi_0, \psi_1$ with $\psi_0 \leq f \leq \psi_1$ and $\mu(\psi_1 - \psi_0) \leq \varepsilon$).

Then, the ergodic means of $f$ converge uniformly, and this implies the existence of $N$ 
such that $\frac1n |S_n^T f(x)| \geq \frac12 |\mu(f)| > 0$ for $n \geq N$ and every $x$.
It follows that the number of visits of $S_n^T f(x)$ to 0 is $\leq N$, for every $x$. By remark \ref{unifx}, $M_n(x) \leq N$, for every $x$, and a fortiori 
(\ref{condi00x}) is satisfied.

{\it Nevertheless, we will see that $(\ref{condi00x})$ may fail in non uniform cases: there are dynamical systems and sets $B$ of positive measure 
such that, for $f = 1_B$, }
\begin{eqnarray}
\limsup_n {M^2_n(T, f, x) \over V_n(T, f, x)} =  1. \label{counterEx}
\end{eqnarray}

\vskip 3mm
\subsection{Counterexamples}

\

In this subsection, we construct a transient counterexample, and also a recurrent counterexample with a function $f$ of null integral such that (\ref{counterEx}) is satisfied.

To construct these counterexamples, we start by considering a general ergodic dynamical system $(X, \mu, T)$ 
and a measurable set $B \subset X$ of positive measure. Let $T_{B}$ be the induced map on $B$, 
$R(x) = R^B(x)= \inf\{k \geq 1: T^k x \in B\}$ the first return time of $x$ in $B$ 
and $R_n(x) = R_n^B(x) := \sum_{k=0}^{n-1} R(T_B^k x)$ the $n$-th return time of $x$ in $B$. 

We take $x \in B$. If $f$ is a function such that $f =0$ outside $B$, the position of the sums up to time $R_{n-1}(x)$ 
are the positions of the ergodic sums $S_n^{T_B} f$ for the induced map up to time $n$, that is:
$$\{f(x), f(x) + f(T_B x), ..., f(x) + f(T_B x) + ... + f(T_B^{n-1} x)\}.$$
For a site $\ell$, the number of visits up to time $R_{n-1}(x)$ of the ergodic sums for $T$ is
$$N_{R_{n-1}(x)}(x,\ell) = \sum_{k= 0}^{n-1} R^B(T_B^k x) \, 1_{S_k^{T_B} f(x) =\ell}$$
and therefore
\begin{eqnarray}
V_{R_{n-1}(x)}(T, x) = \sum_\ell [\sum_{k= 0}^{n-1} R^B(T_B^k x) \, 1_{S_k^{T_B} f(x) =\ell}]^2.\label{VRn}
\end{eqnarray}
{\it Case $f = 1_B$.} Clearly $\sum_{k=0}^{n-1} f(T_B^k x) = n$. For the map $T$, the ergodic sums of $f$ are incremented by 1 when and only 
when the iterates $T^j x$ visit the set B. Otherwise, they stay fixed. The times of visits in $B$, for $x \in B$, are $0, R(x), R(x) + R(T_B x), ... $.
We have:
\begin{eqnarray*}
&&\text{for } x \in B, \sum_{j= 0}^{R_{n-1}(x) + t} f(T^j x) = n, \text{ for } t= 0, ..., R_{n}(x) - R_{n-1}(x) -1.
\end{eqnarray*}
For $N_n(T, x, \ell) = N_n(T, f, x, \ell)$, it holds: 
\begin{eqnarray*}
N_n(T, x, \ell) &=& 0, \text{ if } n <   R_{\ell}(x),\\
&=& t, \text{ if } n = R_{\ell}(x) + t, \text{ with } 0 \leq t <  R_{\ell +1}(x) - R_{\ell}(x),\\
&=& R_{\ell + 1}(x) - R_{\ell}(x) = R(T_B^\ell x), \text{ if } n \geq R_{\ell +1}(x).
\end{eqnarray*}
For $L \geq 1$, we have for the time preceding the $L$-th return to the basis for $f = 1_B$: 
\begin{eqnarray}
M_{R_L(x) - 1}(T, f, x) = \max_{\ell \leq L} R(T_B^\ell x), \ V_{R_L(x) - 1}(T, f, x) = \sum_{\ell \leq L} R^2(T_B^\ell x). \label{visitRet0}
\end{eqnarray}
In order to compute an explicit example, it is easier to start from a given map $S$ and construct a special flow $T$ over this map.

Let $\varphi : X \to \N$ be integrable and $\geq 1$. The (discrete time) special map $T = T_\varphi$ is defined 
\begin{eqnarray*}
&\text{ on } \tilde X := \{(x, k), \ x \in X, k= 0, ..., \varphi(x)-1\} \subset X \times \R, \\
&\text{ by } T(x,k) := (x, k+1), \text{ if } 0 \leq k < \varphi(x)-1, \ :=(S x, 0), \text{ if } k =\varphi(x)-1.
\end{eqnarray*}
Let $\tilde \mu$ be the probability measure defined on $\tilde X$ by $\tilde \mu(A \times \{k\}) = \mu(\varphi)^{-1} \, \mu(A)$, 
for $k \geq 0$ and $A\subset \{x: k \leq \varphi(x)-1\}$. It is $T_\varphi$-invariant.
The space $X$ can be identified with the subset $B= \{(x, 0), x \in X \}$ of $\tilde X$ with normalized measure.
The set $B$ is the basis and $\varphi-1$ the roof function of the special map $T_\varphi$.

As for the map $S$ we will take an ergodic rotation, the special flow $T_\varphi$ will be also ergodic for the measure $\tilde \mu$ on $\tilde X$.

Observe that the recurrence time $R(x) = R^B(x)$ for the special flow in the basis $B$ is $\varphi(x)$ and the $n$-th return time of $x$ in $B$ 
is $R_n(x) = R_n^B(x) = \sum_{k=0}^{n-1} \varphi(S^k x)$. 

For $S$, let us take a rotation $S = S_\alpha$ on $X = \T/ \Z$ by $\alpha \mod 1$, where $\alpha$ is irrational. We denote by $q_n$ the denominators of $\alpha$.
We will construct the measure preserving transformation which is the special flow (with discrete time) over $S_\alpha$ 
with a roof function $\varphi$ such that, for cocycle generated by $1_B$ in the system $(\tilde X, \tilde \mu, T)$,
$\displaystyle \liminf_n {V_n(T, x) \over M_n^2(T, x)} = 1$.

We will use the next lemma with $p= p_n, q = q_n$, the numerators and denominators of $\alpha$.
\begin{lem} \label{orbx} Let $p, q \geq 1$, $(p, q) = 1$, be such that $|\alpha - p/q| < 1/q^2$.
For every $x$, there is a value $0 \leq i < q$ such that $x + i\alpha \mod 1 \in [0,  2/q]$.

More generally, for every interval $I$ of length $2/q$, for every $x$, there is a value $0 \leq i < q$ such that $x + i\alpha \mod 1 \in I$.
\end{lem}
\proof  It is well known that there is exactly one value of $j \alpha \text{ mod }1$, for $0 \leq j < q$, in each interval $[{\ell \over q}, {\ell+1 \over q}[$, $\ell=0,...,q-1$. 
Let us recall a proof. For $j= 0$,  $j \alpha \in [0, 1/q[$. The map $j \to \ell_j = jp \mod q$, which is injective, is a permutation of the set $\{1, ..., q-1\}$ onto itself.
We have $\alpha = p/q + \gamma$, with $|\gamma| < 1/q^2$. 

Assuming $\gamma > 0$, it follows: 
$j\alpha \mod 1 \in [{\ell_j \over q}, {\ell_j \over q} + {j \over q^2}] \subset [{\ell_j \over q}, {\ell_j + 1\over q}[, \text{ for } j = 1, ..., q-1$.
The case $\gamma < 0$ is treated the same way.

Now let us prove the first point. Let $x$ be in $[0, 1[$. There is $i_0 \in \{0,...,q-1\}$ such that $x = {i_0 \over q} + \theta$, with $0 \leq \theta <  1/q$.
By the claim, there is $i \in [0, q[$ such that $i \alpha \mod 1 \in [{q -i_0 \over q}, {q -i_0 + 1 \over q}]$.
Hence $x+i\alpha \mod 1 \in  [\theta, \frac1q + \theta]  \subset [0, \frac2q]$. \eop

Let $(\lambda_n)$ be an increasing sequence of positive integers which will be subjected below to growth conditions.
First we assume that it satisfies the condition: 
\begin{eqnarray}
q_{\lambda_{n+1}} \geq 3 q_{\lambda_n}, \forall n \geq 1. \label{condit3}
\end{eqnarray}
Denote by $J_n$ the interval $\displaystyle J_n = [{3 \over q_{\lambda_{n+1}}}, {3 \over q_{\lambda_n}}]$. For the roof function, we take, with $\varepsilon_n = \frac1{n^2}$,
$$\displaystyle \varphi= 1 + \sum_{n \geq 1} \lfloor \varepsilon_n q_{\lambda_n} \rfloor 1_{J_n}.$$
The function $\varphi$ is integrable: $\int \varphi d\mu \leq 1 + 3 \sum_n \varepsilon_n$.
Observe also that, by (\ref{condit3}), the length of $J_n$ is $> 2/q_{\lambda_n}$ and that $(\varepsilon_n q_{\lambda_n})$ is not decreasing for $n \geq 2$ .

Let $x$ be in the basis. By construction, the orbit of $x$ under the iteration of $T_\varphi$ is that of the rotation $S_\alpha$ until it enters the set $B^c$, 
complementary of $B$ at some time. Then it stays in this set, until it reaches the roof and comes down to the basis. 
Then the dynamic is that of the rotation, until again $S_\alpha^j x$ falls in the set $\varphi > 1$ and so on.

Let $W_n(x)$ be the first visit of $S^j x$ in $J_n$. By lemma \ref{orbx}, we have $W_n(x) \leq q_{\lambda_n}$.

Now we choose $f$ to get a transient counterexample and a recurrent one.

{\it Transient counterexample.}  

We take $f= 1$ on the basis and 0 outside.

The sequence $(\lambda_n)$ is taken such that
\begin{eqnarray}
q_{\lambda_n} \geq n \, (q_{\lambda_{n-1}})^2, n \geq 1. \label{qlambda}
\end{eqnarray}
By (\ref{visitRet0}), we obtain (recall that now $T_B$, the induced map in the basis $B$, is the rotation $S = S_\alpha$ and $R(T_B^j x) = \varphi(S_\alpha^j x)$):
\begin{eqnarray}
M_{R_{W_n(x)}(x) - 1}(T, x) &=& \max_{j \leq W_n(x)} \varphi(S^j x), \label{MRN}\\
V_{R_{W_n(x)}(x) - 1}(T, x) &=& \sum_{j \leq W_n(x)} \varphi^2(S^j x) . \label{VRN} 
\end{eqnarray}
In the above formula, $\varphi(S^j x)$ is either 1 or (for some $k \leq n-1$) $1 + \lfloor \varepsilon_k q_{\lambda_k} \rfloor \leq 1 + \varepsilon_{n-1} q_{\lambda_{n-1}}$,
excepted for the last term which is $1 + \lfloor \varepsilon_n q_{\lambda_n} \rfloor$.

The maximum in (\ref{MRN}) (given by the first visit to $J_n$) is $1 + \lfloor \varepsilon_n q_{\lambda_n} \rfloor \geq \varepsilon_n  q_{\lambda_n}$.
As we have seen, this first visit for the iterates $S^j x$  occurs at a time $\leq q_{\lambda_n}$. It follows by (\ref{qlambda}):
\begin{eqnarray*}
{V_{R_{W_n(x)}(x) - 1}(T, x) \over M^2_{R_{W_n(x)}(x) - 1}(T, x)}
&\leq& q_{\lambda_n} {(\varepsilon_{n-1} \, q_{\lambda_{n-1}})^2 \over (\varepsilon_n \, q_{\lambda_n} - 1)^2} +1 
\leq ({\varepsilon_{n-1} \over \varepsilon_n})^2  \,{(q_{\lambda_{n-1}})^2 \over q_{\lambda_n}}  {1 \over (1 - (\varepsilon_n \, q_{\lambda_n})^{-1})^2} +1 \\
&\leq&  2 \, ({n \over n - 1})^2 \, {(q_{\lambda_{n-1}})^2 \over q_{\lambda_n}} + 1 \leq \frac 4n + 1, \text{ for } n \text{ big enough}.
\end{eqnarray*}
This shows: $\displaystyle \limsup_n {M^2_n(T, f, x) \over V_n(T, f, x)} = 1$. The result is proved for $x$ in the basis $B$, but is satisfied for a.e. $x \in \tilde X$,
since $\displaystyle \limsup_n {M^2_n(T, f, x) \over V_n(T, f, x)}$ is a.e. constant by ergodicity of the special flow and Lemma \ref{MnInv}.

Remark that $S_k f(x) \to +\infty$ for every point $x$.The sequence $(N_n(x, 0))$ is bounded for every $x$, but not uniformly in $x$.

\vskip 3mm
\goodbreak
{\it Recurrent counterexample.}

In order to obtain a recurrent counterexample, we now use a special cocycle over a rotation by $\alpha$ (with $\alpha$ bpq)
studied later (see Subsection \ref{stepf}).

Let $f$ defined on the basis by $f(x) = 1_{[0, \frac12[}(x) - 1_{[\frac12, 1[}(x)$ and 0 outside,
and $S_k f(x) = \sum_{i=0}^{k-1} f(x + i \alpha \mod 1)$.
By (\ref{VRn}), we have
\begin{eqnarray*}
V_{R_{n-1}(x)}(T, f, x) &=& \sum_\ell [\sum_{k= 0}^{n-1} \varphi(x + k \alpha) \, 1_{S_k f(x) =\ell}]^2\\
&=& \sum_\ell [\sum_{k= 0}^{n-1} (1 + \sum_{j} \varepsilon_j q_{\lambda_j} 1_{J_j}(x + k \alpha)) \,  1_{S_k f(x) =\ell}]^2.
\end{eqnarray*}
Observe that for a constant $C$, $1 + \sum_{j < n} \varepsilon_j q_{\lambda_j} 1_{J_j}(x + k \alpha)) \leq C q_{\lambda_{n-1}}$.
Using the bounds for the special function $f$ and $\alpha$ bpq, this implies:
\begin{eqnarray*}
V_{R_{W_{n-1}(x)}}(T, f, x) &\leq& \sum_\ell [\sum_{k= 0}^{q_{\lambda_n}} (1 + \sum_{j < n} \varepsilon_j q_{\lambda_j} 1_{J_j}(x + k \alpha)) \, 1_{S_k f(x) =\ell}]^2\\
&\leq& C^2 \sum_\ell [\sum_{k= 0}^{q_{\lambda_n}} q_{\lambda_{n-1}} \, 1_{S_k f(x) =\ell}]^2 \\
&\leq& C^2 q_{\lambda_{n-1}}^2 \, \sum_\ell [\sum_{k= 0}^{q_{\lambda_n}} \, 1_{S_k f(x) =\ell}]^2
\leq C^2  q_{\lambda_{n-1}}^2 \, q_{\lambda_n}^2 / \sqrt{\log q_{\lambda_n}}.
\end{eqnarray*}

Put $L_n = S_{W_n(x)} f(x)$ for the site visited by the cocycle when $S^j x$ enters $J_n$. We have 
$$M_{R_{W_n(x)}} (T, f, x) \geq N_{R_{W_n(x)}}(T, f, x, L_n(x)) = \varepsilon_n q_{\lambda_n}.$$
Hence:
\begin{eqnarray*}
0 \leq {V_{R_{W_{n}(x)}}(T, f, x) \over M^2_{R_{W_n(x)}}(T, f, x)} - 1
&\leq&  C^2 {q_{\lambda_{n-1}}^2 \, q_{\lambda_n}^2 \over \sqrt{\log q_{\lambda_n}}} {1 \over \varepsilon_n^2 q_{\lambda_n}^2}
=  C^2 {n^4 q_{\lambda_{n-1}}^2 \over \sqrt{\log q_{\lambda_n}}}.
\end{eqnarray*}
Now, we choose a growth condition on $(\lambda_n)$ stronger than (\ref{qlambda}), such that the above bound tends to 0.

This shows the result for $x$ in the basis, hence on the whole space using again Lemma \ref{MnInv}.

\section{\bf Examples} \label{exSect}

In general, for a dynamical system $(X, \mu, T)$ and a cocycle $(T, f)$,  it seems difficult to get a precise estimate of the quantities $N_n(x, \el), M_n(x), V_n(x)$. 
In this section we present two types of cocycles for which this is possible, 
first in the case of strong stochastic properties, in particular for the classical case of random walks, then when they are generated by step functions over rotations.

\subsection{Random walks} \label{sectRW}

\

{\it 1-dimensional cocycle satisfying the LIL.}
 
We start be a remark on the the law of iterated logarithm (LIL). Suppose that $(T, f)$ is a 1-dimensional cocycle which satisfies the LIL. 
Then for a constant $c_1 > 0$, for a.e. $x$, the inequality $|f_n(x)| > c_1 \, (n \ \ln\ln \, n)^\frac12$ is satisfied only for finitely many values of $n$. 
This implies that, for a.e. $x$, there is $N(x)$ such that $|f_n(x)| \leq (c_1 \, n \ \ln\ln \, n)^\frac12$, for $ n \geq N(x)$;
so that, for $N(x) \leq k < n$, $|f_k(x)| \leq (c_1 \, k \ \ln\ln \, k)^\frac12 \leq (c _1\, n \ \ln\ln \, n)^\frac12$. 

Therefore we have $\Card (\Cal R_n(x)) \leq c_2(x) \, (n \ \ln\ln \, n)^\frac12$, with an a.e. finite constant $c_2(x)$.

In dimension 1, by (\ref{growth0}), we get that  for a.e. $x$ there is $c(x) > 0$ such that 
\begin{eqnarray*} 
V_n(x) \geq C(x) \, n^{\frac32} \, (\ln\ln \, n)^{-\frac12}.
\end{eqnarray*}
The case where a LIL is valid includes the case of a 1-dimensional r.w. centered with finite variance, but also the class of cocycles for which a martingale method can be used.

{\it Random walks.} 

Now we consider sequences given by a random walk. 
For random walks in $\Z^d$, the quantities $V_n(x), M_n(x)$ have been studied in many papers since the 50's. 
$M_n(x)$ is called ``maximal multiplicity of points on a random walk'' by Erd\"os  and Taylor \cite{ET60}.
Below, we give a brief survey of several results for r.w.s. First we recall some definitions.

Let $(\zeta_i)_{i \geq 0}$ be a sequence of i.i.d. random vectors on a probability space $(X, \, \mu)$ with values in $\Z^d$ and common probability distribution $\nu$.
The associated {\it random walk} (r.w.) $Z = (Z_n)$ in $\Z^d$ starting from $\0$ is defined by $Z_0 := \0$, 
$$Z_n := \zeta_0 +... + \zeta_{n-1}, n \geq 1.$$
A r.w. can be seen as a special case of cocycle. Indeed, the r.v.'s $\zeta_i$ can be viewed as the coordinate maps on $(X, \, \mu)$ obtained 
as $(\Z^d)^\Z$ equipped with the product measure $\nu^{\otimes \Z}$ and with the shift $T$ acting on the coordinates. 
We have $\zeta_i = \zeta_0 \circ T^i$ and the cocycle relation $Z_{n +n'} = Z_n + Z_{n'} \circ T^n, \forall n, n' \geq 0$.

Let $\cal S := \{\el \in \Z^d: \PP(\zeta_0 = \el) > 0\}$ be the support of $\nu$ and $L$ the sub-lattice of $\Z^d$ generated by $\Cal S$.
Let $D$ be the sub-lattice of $\Z^d$ generated by $\{\el - \el', \el, \el' \in \cal S\}$.

For simplicity (and without loss of generality) in what follows we will assume that the random walk $Z$ is {\it aperiodic} ($L = \Z^d)$.
We exclude also the ``deterministic'' case (i.e., when $\PP(\zeta_0 = \el) = 1$ for some $\el \in \Z^d$) in dimension 1 
(the deterministic case in higher dimension is excluded by aperiodicity).

Notice that all the pointwise limits or bounds mentioned now for random walks are {\it a.s.} statements.
These bounds will show that conditions (\ref{condi000}), (\ref{GC}) are satisfied by  $V_n(x), M_n(x)$ a.s. for on random walks under mild assumptions.

{\it Recurrence/transience.} 

Recall that a r.w. $Z = (Z_n)$ is recurrent if $\sum_{n=1}^\infty \mu(Z_n = \0) = + \infty$ and otherwise transient. 

Recurrence occurs if and only if $\mu(Z_n=\0 \text{ infinitely often}) = 1$, and transience if and only if $\mu(Z_n=\0 \text{ infinitely often}) = 0$ 
(cf. \cite{ChFu51}, \cite{ChOr62}).

For an aperiodic r.w.  $Z$ in dimension $d$ with a moment of order 2 (for $d=1$, a moment of order 1 suffices),
for $d=1, 2$,  $Z$ is recurrent if and only if it is centered. For $d\ge3$, it is always transient.

\goodbreak
{\it Variance.}

Let $(X_\el, \el \in \Z^d)$ be a stationary centered  r.f. with summable correlation and spectral density $\rho$. We have 
$$\frac1n \|\sum_{k=1}^{n-1} X_{Z_k(x)}\|_2^2 = \int_{\T^d} \frac1n |\sum_{k=0}^{n-1} e^{2\pi i \langle Z_k(x), \t \rangle} |^2  \, d\t
= \int_{\T^d} \frac1n K_n(x, \t) \, \rho(\t) \, d\t,$$
where, using (\ref{Vnp}) with $z_k = Z_k(x)$ and $Z_{k}(x) - Z_{j}(x) = Z_{k}(T^j x)$, $\frac1n K_n$ reads
\begin{eqnarray}
\frac1n K_n(x, \t) &&= 1 + 2 \sum_\el \bigl(\sum_{k = 1}^{n-1} \, \frac1n \sum_{j = 0}^{n-k-1} 1_{Z_{k}(T^j x) = \el} \bigr) \, e^{2\pi i \langle \el, t \rangle}. \label{Knl}
\end{eqnarray}
As already recalled, the existence of the asymptotic variance $\lim_n V_n(x)^{-1} \int |\sum_{k=0}^{n-1} X_{Z_k(x)}|^2 d\mu$ has been shown in \cite{CohCo17} 
and the positivity of the limit has been discussed. 

The asymptotic variance may be zero in case a coboundary condition is satisfied.
An interesting situation is that of the sums along a transient (non deterministic) r.w., where the asymptotic variance is always $> 0$.
Below we will recall briefly a proof.

\vskip 6mm
\goodbreak
{\bf Transient case}

For a transient random walk  we use the following general result (Lemma 3.14 in \cite{CohCo17}):
\begin{lem} \label{ergLem} If $(X, \mu, T)$ is an ergodic
dynamical system and $(\varphi_k)_{k \geq 1}$ a sequence of functions in $L^1(X, \mu)$ such that $\sum_{k \geq 1} \|\varphi_k\|_1 < \infty$, then
\begin{eqnarray}
\lim_n \frac1n \sum_{k=1}^{n-1} \sum_{j=0}^{n-k-1} \varphi_k(T^j x) = \sum_{k=1}^\infty \int \varphi_k \, d\mu, \text{ for a.e. } x. \label{ergLemFormula}
\end{eqnarray}
\end{lem}
Therefore, for a transient random walk, we obtain for $\mu$-a.e. $x$:
$$\lim_n {V_n(x)  \over n} = 1  + 2 \lim_n \sum_{k = 1}^{n-1} \, \frac1n \sum_{j = 0}^{n-k-1} 1_{Z_k (T^j x) = \0} 
= 1 + 2 \sum_{k=1}^\infty \mu( Z_k = \0) < + \infty$$
and the normalisation for the variance is by $n$ up to a finite constant factor.

{\it Variance in the non deterministic transient case.}

Now we recall the proof of the positivity of the asymptotic variance.

Let $\Psi(\t) = \E[e^{2\pi i \langle \zeta_0, \t \rangle}], \, \t \in \T^d$.
Observe that $\Psi(\t) \not = 1$ for $\t \not = \0$ in $\T^d$, when the r.w. is aperiodic and $|\Psi(\t)| < 1$, for $\t \not \in  \Gamma_1$,
where $\Gamma_1$ is the closed subgroup $\{\t \in \T^d: e^{2\pi i \langle \r, \t \rangle} = 1, \forall \r\ \in D \}$. 
We put, for $\t \in \T^d \stm0$ and $0 \leq \lambda < 1$,
\begin{align*}
&\Phi(\t) := {1 - |\Psi(\t)|^2 \over |1 - \Psi(\t)|^2} 
= \Re e [{1 + \Psi(\t) \over 1 - \Psi(\t)}],& \\
&\Phi_\lambda(\t) := {1 - \lambda^2 |\Psi(\t)|^2 \over |1 - \lambda \Psi(\t)|^2} 
= -1 + 2 \sum_{k=0}^\infty \lambda^k \Re e(\Psi(\t)^k)
= -1 + 2 \sum_{k=0}^\infty \lambda^k \mu(Z_k = \el) \, \cos(2\pi \langle \el, \t\rangle),& 
\end{align*}
where the last relation follows from
$\Re e(\Psi(\t)^k) = \Re e(\E[e^{2\pi i \langle Z_k, \t \rangle}]) = \sum_\el \mu(Z_k = \el) \, \cos(2\pi \langle \el, \t \rangle)$.

We put $\Phi(\0) = 0$.The function $\Phi$ is even, non-negative and $\Phi(\t) = 0$ only on $\Gamma_1$, which is $\not = \T^d$ when the r.w. is non deterministic
(if the r.w. is deterministic, $\mu(\zeta_0 = \el) = 1$ for some $\el \in \Z^d$ and this implies $|\Psi(\t)| \equiv 1$, but this case is excluded).
Therefore $\Phi$ is $\not = 0$ a.e. for the Lebesgue measure on $\T^d$. 

\begin{proposition} (cf. \cite{Sp64}) Let $Z =(Z_n)$ be a transient aperiodic random walk in $\Z^d$.
There is a non-negative constant $M$ such that the Fourier coefficients of $\frac1n K_n$ converges to those of $\Phi +M \delta_\0$
and $\displaystyle \lim_n \int \frac1n K_n \, \rho \, d\t > 0$.
\end{proposition}
\proof We use that, if $(Z_n)$ is a transient,  for all $\el \in \Z^d$, we have $\sum_{k = 1}^{\infty} \, \mu(Z_k = \el) < +\infty$.

Therefore, the series $I(\el) := -1_{\el = \0} + \sum_{k = 0}^{\infty} \,  [\mu(Z_k = \el) + \mu(Z_k = - \el)]$ converges and
by (\ref{Knl}) and Lemma \ref{ergLem}, the even functions $\frac1n K_n(x, .)$ satisfy:
\begin{eqnarray*}
\int_{\T^d} \frac1n K_n(x, .) \, \cos 2\pi \langle \el, . \rangle \, d\t &&= -1_{\el = \0} + \sum_{k = 0}^{n-1} \, 
\frac1n \sum_{j = 0}^{n-k-1} [1_{Z_{k}(T^j x) = \el} + 1_{Z_{k}(T^j x) = -\el}] {\underset{n \to \infty} \to } \ I(\el).
\end{eqnarray*}
Note that above the sum over $k$ is written starting from 0. By letting $n$ tend to infinity in the relation 
\begin{eqnarray*}
&& - 1_{\el= \0} + \sum_{k = 0}^{\infty} \, \lambda^k [\mu(Z_k = \el) + \mu(Z_k = - \el)] \\
= && \int_{\T^d} \cos 2\pi \langle \el, . \rangle \, [-1 +2 \Re e ({1 \over 1 - \lambda \Psi(.)})] \, d\t \ = \int_{\T^d} \cos 2\pi \langle \el, \t\rangle \, \Phi_\lambda(.) \, d\t,
\end{eqnarray*}
we get since the left sum tends to $I(\el)$:
\begin{eqnarray*}
I(\el) = \lim_{\lambda \uparrow 1} \int_{\T^d} \cos 2\pi \langle \el, \t\rangle \, \Phi_\lambda(\t) \, d\t. 
\end{eqnarray*}
Taking $\el = \0$ in the previous formula, it follows from Fatou's lemma:
\begin{eqnarray*}
&&I(\0) = 1 + 2 \sum_{k = 1}^{\infty} \, \mu(Z_k = \0) = \lim_{\lambda \uparrow 1} \int_{\T^d} \, \Phi_\lambda(\t) \, d\t 
\geq \int_{\T^d} \lim_{\lambda \uparrow 1} \Phi_\lambda(\t) \, d\t = \int_{\T^d} \Phi(\t) \, d\t.
\end{eqnarray*}
This shows the integrability of $\Phi$ on $\T^d$ and we can write with a constant $M \geq 0$
\begin{eqnarray*}
&&I(\0) = \lim_{\lambda \uparrow 1} \int_{\T^d} \, \Phi_\lambda(\t) \, d\t  = \int_{\T^d} \lim_{\lambda \uparrow 1} \Phi_\lambda(\t) \, d\t + M
=\int_{\T^d} \Phi(\t) \, d\t + M.
\end{eqnarray*}
Let $U_\eta$ be the ball of radius $\eta > 0$ centered at $\0$. By aperiodicity of the r.w., $\Psi(\t) \not = 1$ for $\t$ in $U_\eta^c$, the complementary in $\T^d$ of $U_\eta$, 
This implies $\sup_{\t \in U_\eta^c} \sup_{\lambda < 1} \Phi_\lambda(\t) < +\infty$.

Therefore, we get: $\displaystyle \lim_{\lambda \uparrow 1} \int_{U_\eta^c} \cos 2\pi \langle \el, \t\rangle \,  \Phi_\lambda(\t) \,
d\t = \int_{U_\eta^c} \cos 2\pi \langle \el, \t\rangle \,  \Phi(\t) \, d\t$, hence:
\begin{align*}
&I(\el) = \int_{U_\eta^c} \cos 2\pi \langle \el, \t\rangle \, \Phi(\t)\, d\t + \lim_{\lambda \uparrow 1} \int_{U_\eta} \cos 2\pi \langle \el,\t\rangle \, \Phi_\lambda(\t) \, d\t, 
\, \forall \eta>0, &  
\end{align*}
which can be be written:
\begin{align}
&- \int_{U_\eta} \cos 2\pi \langle \el, . \rangle \, \Phi\, d\t = I(\el) - \int_{\T^d} \cos 2\pi \langle \el, . \rangle \, \Phi\, d\t 
- \lim_{\lambda \uparrow 1} \int_{U_\eta} \cos 2\pi \langle \el, . \rangle \, \Phi_\lambda \, d\t.& \label{IpLim}
\end{align}
Let $\varepsilon > 0$. By positivity of $\Phi_\lambda$, we have, for $\eta(\varepsilon)$ small enough:
\begin{eqnarray*}
&&(1 - \varepsilon) \int_{U_{\eta(\varepsilon)}} \, \Phi_\lambda \, d\t \leq \int_{U_{\eta(\varepsilon)}} \cos 2\pi \langle \el,
. \rangle \, \Phi_\lambda \, d\t \leq  (1+\varepsilon) \int_{U_{\eta(\varepsilon)}} \,  \Phi_\lambda\, d\t;
\end{eqnarray*}
By subtracting $\int_{U_\eta(\varepsilon)} \cos 2\pi \langle \el, \t\rangle \, \Phi(\t)\, d\t$ in the previous inequalities and (\ref{IpLim}), we get:
\begin{eqnarray*}
&&(1 - \varepsilon) \int_{U_{\eta(\varepsilon)}} \, \Phi_\lambda \, d\t - \int_{U_{\eta(\varepsilon)}} \cos 2\pi \langle \el, . \rangle \, \Phi\, d\t \\
&&\leq  I(\el) - \int_{\T^d} \cos 2\pi \langle \el, . \rangle \, \Phi \, d\t  - \lim_{\lambda \uparrow 1} \int_{U_\eta(\varepsilon)} \cos 2\pi \langle \el, . \rangle \, \Phi_\lambda \, d\t
+ \int_{U_{\eta(\varepsilon)}} \cos 2\pi \langle \el, . \rangle \, \Phi_\lambda \, d\t \\
&&\leq  (1+\varepsilon) \int_{U_{\eta(\varepsilon)}} \,  \Phi_\lambda \, d\t - \int_{U_{\eta(\varepsilon)}} \cos 2\pi \langle \el, . \rangle \,\Phi \, d\t ;
\end{eqnarray*}
As we can chose $\lambda$ such that 
$$|- \lim_{\lambda \uparrow 1} \int_{U_\eta(\varepsilon)} \cos 2\pi \langle \el, . \rangle \, \Phi_\lambda \, d\t
+ \int_{U_{\eta(\varepsilon)}} \cos 2\pi \langle \el, . \rangle \, \Phi_\lambda \, d\t| \leq \varepsilon,$$
we obtain:
\begin{eqnarray*}
&& -\varepsilon + (1 - \varepsilon) \int_{U_{\eta(\varepsilon)}} \, \Phi_\lambda \, d\t - \int_{U_{\eta(\varepsilon)}} \cos 2\pi \langle \el, . \rangle \,\Phi\, d\t  \\
&& \leq I(\el) - \int_{\T^d} \cos 2\pi \langle \el, .\rangle \, \Phi \, d\t 
\leq\varepsilon + (1 + \varepsilon) \int_{U_{\eta(\varepsilon)}} \, \Phi_\lambda \, d\t - \int_{U_{\eta(\varepsilon)}} \cos 2\pi \langle \el, . \rangle \, \Phi\, d\t
\end{eqnarray*}
For $\varepsilon$ small enough, $\int_{U_{\eta(\varepsilon)}} \cos 2\pi \langle \el, . \rangle \, \Phi \, d\t$ 
can be made arbitrary small,  as well as $\varepsilon \sup_{\lambda < 1} \int_{U_\eta} \Phi_\lambda \, d\t$, 
since $\Phi$ is integrable and $\sup_{\lambda < 1} \int_{\T^d} \Phi_\lambda \, d\t < \infty$.

This shows that $I(\el) - \int_{\T^d} \cos 2\pi \langle \el, . \rangle \, \Phi \, d\t - \int_{U_{\eta(\varepsilon)}} \, \Phi_\lambda \, d\t$  
can be made arbitrarily small for $\varepsilon > 0$ small and $\lambda$ close to 1.
The same is true for $\el =0$ and also for the difference
\begin{eqnarray*}
&[I(\el) - \int_{\T^d} \cos 2\pi \langle \el, . \rangle \, \Phi \, d\t - \int_{U_{\eta(\varepsilon)}} \, \Phi_\lambda \, d\t]
- [I(\0) - \int_{\T^d} \, \Phi \, d\t - \int_{U_{\eta(\varepsilon)}} \, \Phi_\lambda \, d\t]\\
&= [I(\el) - \int_{\T^d} \cos 2\pi \langle \el, . \rangle \, \Phi \, d\t] - [I(\0) - \int_{\T^d} \, \Phi \, d\t]
= [I(\el) - \int_{\T^d} \cos 2\pi \langle \el, . \rangle \, \Phi \, d\t ] - M].
\end{eqnarray*} 
Therefore $I(\el) = \int_{\T^d} \cos 2\pi \langle \el, \t\rangle \, \Phi(\t)\, d\t + M$ for all $\el$
and the Fourier coefficients of $\frac1n K_n$ converges to those of $\Phi +M \delta_\0$.

As the non-negative sequence $(\frac1n K_n)$ is bounded in $L^1$-norm and the density $\rho$ is continuous, 
this proves $\displaystyle \int \frac1n K_n \rho \, d\t \to \int \Phi \rho \, d\t + M \rho(\0)$.
Moreover, the limit is $> 0$ since both $\Phi$ and $\rho$ are not 0 a.e.

It is shown in \cite{Sp64} that $M= 0$ for $d > 1$. \eop

\vskip 3mm
{\it Behaviour of $M_n(x)$.}

In the transient case ($d \geq 3$) (at least for a simple r.w.),  Erd\"os  and Taylor (1960) proved that for a constant $\gamma > 0$ depending on the dimension,
$$\lim_n {M_n(x) \over \log n} = \gamma.$$

{\bf Recurrent case}

In dimension 1, H. Kesten has shown that $\displaystyle \limsup_n {M_n \over \sqrt{ n \, \ln \ln n}} = \sqrt 2 / \sigma$.
Therefore in dimension 1, we have the following lower and upper bounds for $V_n$:
$$ C_1(x) \, n^{\frac32} \, (\ln\ln \, n)^{-\frac12} \leq V_n(x) \leq C_2(x) \, n^\frac32 (\ln \ln n)^\frac12.$$

{\it Dimension $d=2$.}

There is a deterministic rate (law of large numbers): for a constant $C_0$.
$${\int V_n \, d\mu \over n \log n} \to C_0 \text{ and } {V_n(x) \over n \log n} \to C_0,  \text{ for a.e. } x.$$
For a planar simple random walk, Erd\"os and Taylor \cite{ET60} have shown:
\begin{eqnarray} \limsup_n \, {M_n(x) \over (\log n) ^2} \leq \frac1\pi. \label{DPRZbound}
\end{eqnarray}

The result has been extended by Dembo, Peres, Rosen and Zeitouni \cite{DPRZ01},
who proved for an aperiodic centered random walk on $\Z^2$ with moments of all orders:
\begin{eqnarray*} 
\lim_n  {M_n(x) \over (\log n) ^2} = \frac1{2 \pi \det (\Gamma)^\frac12},
\end{eqnarray*}
where $\Gamma$ is the covariance matrix associated to the random walk.

As shown in the proof in \cite{DPRZ01}, it suffices to suppose that the 2-dimensional r.w. is aperiodic.
Moreover, the proof for the upper bound is based on the local limit theorem which uses only the existence of the moment of order 2.
Therefore, assuming the existence of the moment of order 2, the upper bound (\ref{DPRZbound}) holds.

It follows in this case: there exist $C(x)$ a.e finite such that:
\begin{eqnarray*} 
{M_n^2(x) \over V_n(x)} \leq C(x) {(\log n)^3 \over n}. 
\end{eqnarray*}

\subsection{Extensions of the r.w. case}

\ 

{\it 1) Consequence of the Local Limit Theorem (LLT).}

The Local Limit Theorem, when it is satisfied by the cocycle $(T, f)$, gives some pointwise information on $V_n(x)$.
For example, if $d=2$, the following lemma holds:
\begin{lem} Suppose that the LLT holds and $d= 2$. Then, for every $\varepsilon > 0$, 
there is an integrable function $C$ (depending on $\varepsilon$), such that: 
\begin{eqnarray}
&&N_n(x, 0) \leq C(x) (\ln n)^{2 + \varepsilon}, \ V_n(x) \leq C(x) \, n \, (\log n)^{2 + \varepsilon}. \label{Tn}
\end{eqnarray}
\end{lem}
\proof  By the LLT, it holds, for $n \geq 1$,
$$\int  N_{n}(. , 0) d\mu = \sum_{k=1}^n \int 1_{f_k = 0} \, d\mu \leq C \sum_{k=1}^n {1 \over k} \leq C \ln n.$$
Let $\varepsilon$ be a positive constant. Putting $\Gamma(x) =  \sum_{n= 1}^\infty n^{-(2+\varepsilon)} N_{2^n}(x, 0)$, we have:
$$\int \Gamma(x) \, d\mu(x) \leq C \sum_{n= 1}^\infty n^{-(2+\varepsilon)} n 
=C \sum_{n= 1}^\infty n^{-(1+\varepsilon)} < +\infty,$$
so that $N_{2^n}(x, 0) \leq \Gamma(x) \, n^{2 + \varepsilon}$, where $\Gamma$ is integrable.

If $2^{k_n} \leq n < 2^{(k_n+1)}$, then with $p=2+\varepsilon$, we have for $n$ big enough,
$${N_n(x, 0) \over (\log_2 n)^p} \leq {N_{2^{(k_n+1)}}(x, 0)  \over k_n^{p}}
\leq {(k_n+1)^{p}\over k_n^{p}} {N_{2^{(k_n+1)}}(x, 0)  \over (k_n+1)^{p}}
= (1+1/k_n)^{p} \, \Gamma(x) \leq 2 \Gamma(x).$$
For $V_n(x)$,  by (\ref{expVnx}) we have:
$$\int V_n(x) \, d\mu(x) = 2 \sum_{k=1}^{n-1} \int N_{n-k}(x, 0) \, d\mu(x) + n = O(n \log n).$$
As above for $N_n(x, 0)$, the pointwise bound (\ref{Tn}) follows.
\eop

Among example of cocycles satisfying a LLT, there are the r.w.'s (but with more precise results as recalled above),
but also cocycles generated by functions with values in $\Z^d$ depending on a finite number of coordinates over a sub-shift of finite type endowed with a Gibbs measure
(\cite{GuiHar88}), (\cite{Gou05}). 

\vskip 3mm
{\it 2) Functions depending on a finite number of coordinates on a Bernoulli scheme.}

Now we try to bound $M_n(x)$ in situation which extends slightly that of random walks.

Suppose that  $(X, \mu, T)$ is a Bernoulli scheme with $X = I^\N$, where $I$ is a finite set. Let $f: x \to f(x_1, ..., x_r)$ be a centered function 
from $X$ to $\Z^d$, $d \geq 1$, depending on a finite number of coordinates.

Let us consider the generalized random walk $(Z_n)$ defined by the sequence of ergodic sums $Z_n(x) = f_n (x) = X_0(x) + ... + X_{n-1}(x)$,
where $X_k(x) =f(T^{k} x)$.
\begin{lem} For all $m \geq 1$ and for constants $C_m, C_m'$ independent of $\el$, we have:
\begin{eqnarray*}
\int N_n^m(.,\el) \, d\mu = \int [\sum_{k=1}^{n} 1_{f_k = \el}]^m \, d\mu &&\le C_m n^{m/2}, \for d=1,\\
&& \le C_m' (\Log n)^m, \for d=2. \label{majERnm2}
\end{eqnarray*}
\end{lem}
\proof \ 
We bound the sum
$\sum_{1 \leq k_1 < k_2 < ... < k_m \leq n} \mu(f_{k_1} = \el, f_{k_2} = \el, ..., f_{k_m} = \el)$.

For $r \leq k_1 < k_2 < ... < k_m \leq n$, writing $X_k$ instead of $T^k f$, we have:
\begin{eqnarray*}
&& \mu(X_0+ ...+X_{k_1-1} = \el, X_{k_1} + ...+X_{k_2-1} = \0, ...,  X_{k_{m-1}} + ...+X_{k_m-1} = \0) \\
&&= \sum_{a_{1,1},..., a_{r,1}, ..., a_{1,m},..., a_{r,m} \in S} \ \mu[ \\
&&X_0+ ...+X_{k_1-r} = \el - (a_{1,1} + ...+ a_{r,1}), X_{k_1-r +1}= a_{1,1,} ..., X_{k_1-1}= a_{r,1},  \\
&&X_{k_1} + ...+X_{k_2-r} =  - (a_{1,2} + ...+ a_{r,2}), X_{k_2-r +1}= a_{1,2},... , X_{k_2-1}= a_{r,2}, ...\\
&&  X_{k_{m-1}} + ...+X_{k_m-r} =- (a_{1,m} + ...+ a_{r,m}), X_{k_m-r +1}= a_{1,m},... , X_{k_m-1}= a_{r,m}];
\end{eqnarray*}
which is less than:
\begin{eqnarray*}
&&\sum_{a_{1,1},..., a_{r,1}, ..., a_{1,m},..., a_{r,m} \in S} \ \mu[X_0+ ...+X_{k_1-r} = \el - (a_{1,1} + ...+ a_{r,1}), \\
&&X_{k_1} + ...+X_{k_2-r} =  - (a_{1,2} + ...+ a_{r,2}), ..., \ X_{k_{m-1}} + ...+X_{k_m-r} =- (a_{1,m} + ...+ a_{r,m})].
\end{eqnarray*}
Now inside $[ . ]$ the events are independent. By independence and stationarity, the preceding sum is
\begin{eqnarray*}
\sum_{a_{1,1},..., a_{r,1}, ..., a_{1,m},..., a_{r,m} \in S} &&\mu[X_0+ ...+X_{k_1-r} = \el - (a_{1,1} + ...+ a_{r,1})]\\
&&\mu[X_{0} + ...+X_{k_2- k_1 - r} =  - (a_{1,2} + ...+ a_{r,2})] \ ... \\ 
&&\mu[\ X_{0} + ...+X_{k_m - k_{m-1}-r} =- (a_{1,m} + ...+ a_{r,m})].
\end{eqnarray*}

With $\tau_n(k) = \sup_\el \mu(X_0+ ...+X_{k-1} = \el) \, 1_{[0, n](k)}$, we get the following bound 
\begin{eqnarray*}
&&\sum_{1 \leq k_1 < k_2 < ... < k_m \leq n} \mu(f_{k_1} = \el, f_{k_2} = \el, ..., f_{k_m} = \el) \\
&&= \sum_{r \leq k_1 < k_2 < ... < k_m \leq n} \mu(X_1+ ...+X_{k_1-1} = \el, X_{k_1} + ...+X_{k_2-1} = \0, ...,  X_{k_{m-1}} + ...+X_{k_m-1} = \0) \\
&&\leq s^{rm} \, \sum_{r \leq k_1 < k_2 < ... < k_m \leq n} \tau_n({k_1-r}) \, \tau_n({k_2- k_1 - r})\,  ... \, \tau_n({k_m - k_{m-1}-r})\\
&&\leq C s^{rm} \sum_k (\tau_n * \tau_n * ... * \tau_n)(k) \leq C s^{rm} (\sum_k \tau_n(k))^m.
\end{eqnarray*}
We have $s^{rm}$ terms for the first sum, where the cardinal $|S|$ is denoted by $s$,.

Now we can use, as for the usual r.w., convolution and the local limit theorem. \eop

From the lemma, it follows easily in the recurrent case that  for a.e. $x$, for all $\varepsilon>0$,
$$ \text{ if } d=1, M_n(x) = o(n^{\frac12+\varepsilon}) \ \text{and, if } d=2, M_n(x) = o(n^\varepsilon).$$
In the transient case, if there is a moment of order $\eta$ for some $\eta>0$,  then $M_n(x) =o(n^{\varepsilon})$ for all $\varepsilon>0$.
For these estimates, in both cases, see \cite{CohCo17}.

A question if to extend the previous results to a larger class of functions depending weakly on the far coordinates. For such an extension is that
explicit bounds in the LLT are not always available.

\subsection{Step functions over rotations} \label{stepf}

\

Now we take $X = \mathbb T^r$, $r \geq 1$ endowed with $\mu$, the uniform measure and we consider cocycles over rotations. 
When they are centered, such cocycles are strongly recurrent and therefore the associated quantities $V_n$ and $M_n$ are big. 
The difficult part is to bound them from above. We will give an example where an upper bound can be obtained.

Let $T_\alpha$ be the rotation by an irrational $\alpha$. For $f: X \to \Z^d$, recall that the cylinder map (cf. Subsection \ref{genCocy1}) is 
$\tilde T_{f, \alpha} = \tilde T_\alpha: X  \times \Z^d \to X \times \Z^d$ defined by $\tilde T_\alpha(x,\el)=(x+\alpha, \el+f(x))$.
 
\vskip 3mm
{\it Non centered step cocycles over a rotation.}

Let $f$ be a non centered function with a finite number of values values in $\Z^d$.
Suppose that $f$ is Riemann integrable, which amounts to assume that, for the uniform measure of the torus, 
the measure of the set of discontinuity points of $f$ is zero.

Then by a remark in Subsection \ref{noncent}, $M_n(x)$ is bounded uniformly in $x$ and $n$.
Therefore, for $V_n(x)$, the bounds $n \leq V_n(x) \leq C n$ are satisfied.

\vskip 3mm
{\it Centered step cocycles over a 1-dimensional rotation.}

The interesting situation is that of centered functions. We will consider the case $r = 1$
and when the irrational number $\alpha$ has {\it bounded partial quotients}.
 
Recall that an irrational $\alpha$ with continued fraction expansion $[0; a_1, a_2, ..., a_n, ...]$ is said to have bounded partial quotients (bpq) 
if $\sup_n a_n < +\infty$. The set of bpq numbers has Lebesgue measure zero and Hausdorff dimension 1.

In the sequel of this subsection, $\alpha$ will be an irrational bpq number (for instance a quadratic irrational) and $f$ a centered function with values in $\Z$ and bounded variation.

By Denjoy-Koksma inequality, there is a logarithmic bound for the cocycle $(T_\alpha, f)$: $|f_n(x)| \leq C \ln n$, for a constant $C$. 

The cocycle is strongly recurrent to 0 (and this is true for $d \geq 1$ if $f$ centered has values in $\Z^d$, when its components have bounded variation).
This makes the corresponding maximum $M_n(x)$ big. 
Nevertheless, we will see that condition (\ref{condi00}) is satisfied, at least for a special example.

\vskip 3mm
{\bf Lower bound.}

{\it Lower bound for $V_n$ and variance, case $d=1$.} 

For a general sequence $(z_k)$, we can obtain a lower bound for $V_n$ by an elementary method when there is an upper bound for the variance defined below.
\begin{lem} Defining the mean $m_n$ and the variance $\sigma_n^2$ by
$$m_n = \frac1n \sum_{k=1}^n z_k, \ \sigma_n^2 = \frac1n \sum_{k=1}^n (z_k - m_n)^2,$$
we have
\begin{eqnarray}
V_n \geq \frac19 \, {n^2 \over \sigma_n}, \text{ if } \sigma_n > 1. \label{minorVn}
\end{eqnarray}
\end{lem}
\proof Suppose that $\sigma_n > 0$. For $\lambda > 1$, let $\Delta_\lambda := [-\lambda \sigma_n + m_n, \, \lambda \sigma_n +m_n] \bigcap \Z$.
We have:
$$\sigma_n^2 \geq \frac1n \sum_{k=0}^{n-1} (z_k - m_n)^2 1_{z_k \in \Delta_\lambda^c} 
\geq \frac1n \sum_{k=0}^{n-1} (1_{z_k \in \Delta_\lambda^c}) \, \lambda^2 \sigma_n^2.$$
Therefore: $\sum_{k=0}^{n-1} 1_{z_k \in \Delta_\lambda} \geq n (1 - \lambda^{-2})$. As $\Card(\Delta_\lambda) \leq 2 \lambda \sigma_n +1$.
It follows by (\ref{majAn}):
$$V_n \geq {(1 - \lambda^{-2})^2 \over 2 \lambda \sigma_n +1} \, n^2.$$
For $\lambda = 2$ we get: $V_n \geq  \ {\frac9{16} \over 4 \sigma_n +1} \, n^2 \geq \frac 9{80} \, {n^2 \over \sigma_n}, \text{ if } \sigma_n > 1$;
hence (\ref{minorVn}).  \eop

\vskip 3mm
If $z_k$ is given by ergodic sums, i.e., $z_k=  f_k(x)$ , let 
$$m_n(x) := \frac1n \sum_{k=1}^n f_k(x), \ \sigma_n^2(x) = \frac1n \sum_{k=1}^n (f_k(x) - m_n(x))^2.$$
By \cite[Proposition 13]{Bo21}, for $\alpha$ bpq and $f$ with bounded vartion, it holds $\sigma_n^2(x) \leq C \ln n$.

Using (\ref{minorVn}) and $V_n(x) \leq n M_n(x)$, this gives a lower bound for $V_n(x)$ and $M_n(x)$:
\begin{eqnarray}
V_n(x) \geq c \, {n^2 \over \sqrt{\ln n}}, \  M_n(x) \geq c \, {n \over \sqrt{\ln n}}. \label{minorRot1}
\end{eqnarray}

Below we will get an estimate from above in the following example.

\ex \label{ex1} $f={\bf 1}_{[0,\frac12)}-{\bf 1}_{[\frac12, 1)}$ and $\alpha$ bpq.

{\it Upper bound  for the example (\ref{ex1}).}

For $f$ as above and $\alpha$ bpq, we have by \cite{ABN17}, for some constant $C_1>0$,
\begin{eqnarray}
\|N_n(\cdot, 0)\|_{\infty}= \|\tilde S_n ({\bf 1}_{\mathbb T^1\times \{0\}}) (\cdot, 0)\|_\infty \leq \frac{C_1 n}{\sqrt{\log n}}. \label{majUnif0}
\end{eqnarray}
Remark that the bound (\ref{majUnif0}) is obtained in \cite{ABN17} as the limit of $\|N_n(\cdot, 0)\|_p$, the $L^p$-norm of $N_n(\cdot, 0)$, as $p$ goes to $\infty$. 
Therefore the bound holds for the norm $\| . \|_{ess\,sup}$, but it can be easily replaced by the uniform norm as written above. 
Indeed, for any $x$, there is a neighborhood $V(x)$ of $x$, such that for $y \in V(x)$, $|N_n(x, 0) - N_n(y, 0)| \leq 1$ (at most one jump in $V(x)$).
As one can find $y \in V(x)$ satisfying $N_n(y, 0) \leq \frac{C_1 n}{\sqrt{\log n}}$, the same inequality holds for $x$, with $C_1$ replaced by $2 C_1$.

Using Remark \ref{unifx}, it follows:
\begin{eqnarray}
M_n(x) \leq C_1 \frac{n}{\sqrt{\log n}}. \label{maxRot1}
\end{eqnarray}

By (\ref{maxRot1}) and since $V_n(x) \leq n \, M_n(x)$, we obtain 
\begin{eqnarray}
V_n(x) \leq C_1 \frac{n^2}{\sqrt{\log n}}. \label{majVn00}
\end{eqnarray}

From (\ref{minorRot1}), (\ref{maxRot1}) and (\ref{majVn00}), it follows: $V_n(x) \asymp n^2/ \sqrt{\log n}$ and $M_n(x) \asymp n/ \sqrt{\log n}$,
where $a_n \asymp b_n$ for two sequences $(a_n)$ and $(b_n)$ means $c \, a_n \leq b_n \leq C \, a_n, \forall n \geq 1$, with two positive constants $c, C$.

Therefore we get in this special example \ref{ex1}:
\begin{eqnarray}
{M_n^2(x) \over V_n(x)} \leq (\frac{C_1 n}{\sqrt{\log n}})^2 /  \frac{c n^2}{\sqrt{\log n}} = {C_1^2 \over c} \, \frac{1}{\sqrt{\log n}} \to 0. \label{majVn0}
\end{eqnarray}
Condition (\ref{condi00}) of Theorem \ref{empirThm1} is satisfied in this example, as well as the condition of Theorem \ref{indicat} a), 
hence a Glivenko-Cantelli theorem along $(S_nf(x))$ for i.i.d. r.v.'s.

But the sufficient conditions for the Glivenko-Cantelli theorems \ref{ratethm}, \ref{indicat} b), \ref{PDQemp} are not satisfied by this cocycle and more
generally, in view of the lower bound (\ref{minorRot1}), by a cocycle defined by step functions over a bpq irrational rotation.

\vskip 3mm
\section{\bf About limit theorems along ergodic sums} \label{sectGC0}

\subsection{Glivenko-Cantelli theorem along ergodic sums}

\

The Glivenko-Cantelli theorem recalled in the introduction is a (pointwise) law of large numbers uniform over a set of functions (here the indicators of intervals).
When the r.v.'s $X_k$ are i.i.d., the proof is an easy consequence of the strong law of large numbers applied to the sequence of i.i.d. bounded r.v.'s ($1_{X_k \leq s})$.
Using Birkhoff's ergodic theorem, the Glivenko-Cantelli theorem has been extended to the setting of a strictly stationary sequence $(X_k)$ 
of random variables. More precisely, formulated in terms of dynamical systems, the following holds:

Let $(Y, {\cal A}, \nu)$ be a probability space and $S$ an ergodic measure preserving transformation on $Y$. For any measurable function
$\varphi: Y \to \R$, let us consider the strictly stationary sequence $(X_k)$ defined by $X_k = \varphi \circ S^k, k \geq 0$. 
Then the sequence of empirical distribution functions satisfies: 
$\text{ for } \nu \text{ a.e. } y \in Y, \sup_s |\frac1n \sum_{k=0}^{n-1}  1_{X_(y)k \leq s} - F(s)| \to 0$, where $F(s) = \nu(\varphi \leq s)$. 

Observe that the result is an application of Birkhoff's theorem and Lemma \ref{ChungLem} recalled in Section \ref{sectGen}.
Its extension  to the non ergodic case has been formulated by Tucker \cite{Tu59}, the distribution function $F(s)$ being replaced 
by the conditional distribution function $\E(1_{\varphi \leq s} | {\Cal J})$, where ${\cal J}$ is the $\sigma$-algebra of $S$-invariant sets.
In others words, we have:
$$\text{ for } \nu \text{ a.e. } y \in Y, \ \lim_{n \to \infty} \, \sup_s |\frac1n \sum_{k=0}^{n-1} 1_{\varphi(S^k y) \leq s} -  \E(1_{\varphi \leq s} | {\Cal J})(y)| = 0.$$
The above formula relies on the ergodic decomposition which can be used in the proof.

In the previous framework, for a process,  a Glivenko-Cantelli like theorem sampled along a sequence generated by a dynamical system can be obtained as follows:

As in Subsection \ref{genCocy}, let $T$ be an ergodic measure preserving transformation on a probability space $(X, {\cal B}, \mu)$ 
and $f$ a measurable function on $X$ with values in $\Z^d$, $d \geq 1$. 

Let us take a second system $(\Omega, \PP, \theta)$, where $\theta = (\theta^\el)_{\el \in \Z^d}$ is a $\Z^d$-action preserving $\PP$.

The skew product associated to the cocycle $(T, f)$ and $\theta$ is the map: $T_{\theta, f}: (x, \omega) \to (Tx, \theta^{f(x)}\omega)$ from $X \times \Omega$ to itself. 
By iteration we get: 
$$T_{\theta, f}^k(x, \omega) = (T^k x, \theta^{f_k(x)}\omega).$$ 
For example, as $\Z^d$-action, we can take a $\Z^d$-Bernoulli shift $(\Omega, \PP, (\theta^\el)_{\el \in \Z^d})$, 
with $\PP$ a product measure and $\theta$ the shift on the coordinates.
If $X_0$ is the first coordinate map, then $(X_\el) = (X_0 \circ \theta^\el)$ is a family of i.i.d. r.v.'s indexed by $\Z^d$.

In general, let ${\Cal I}_{\theta, f}$ denote the conditional expectation with respect to the $\sigma$-algebra of $T_{\theta, f}$-invariant sets.
The ergodic theorem for $T_{\theta, f}$ shows that, for $\psi \in L^1(\mu \times \PP)$,
\begin{eqnarray}
\lim_n \frac1n \, \sum_{k=0}^{n-1} \psi(T^k x, \theta^{f_k(x)}\omega) = {\Cal I}_{\theta, f}(\psi)(x, \omega), 
\text { for } \mu \times \PP \text {-a.e.} (x, \omega). \label{lim0}
\end{eqnarray}
If $\varphi$ is a measurable function on $\Omega$, putting $\psi_s(x, \omega) = {\bf 1}_{I_s} (\varphi(\omega))$, where $I_s$ is the half-line $]-\infty, s]$, we have 
$$\psi_s(T_{\theta, f}^k (x, \omega)) = {\bf 1}_{I_s} (\varphi(\theta^{f_k(x)} \omega)).$$
By the quoted Tucker's result, the convergence in (\ref{lim0}) for each $\psi_s$, $s \in \R$, can be strengthened into a uniform convergence with respect to $s$:
\begin{eqnarray*}
&\text{ for $\mu \times \PP$-a.e } (x, \omega), \, \frac1n \, \sup_s |\sum_{k=0}^{n-1} \, {\bf 1}_{I_s} (\varphi(\theta^{f_k(x)}\omega)) - \Cal I(\psi_s)(x, \omega)| \to 0. 
\end{eqnarray*}
Therefore, by the Fubini theorem, there is a ``sampled'' version of the Glivenko-Cantelli theorem for the empirical process of a stationary sequence:
\begin{proposition}  For $\mu$-a.e  $x$, we have
\begin{eqnarray*}
&|\sup_s \frac1n \, \sum_{k=0}^{n-1} \, {\bf 1}_{I_s} (\varphi(\theta^{f_k(x)}\omega)) - \Cal I(\psi_s)(x, \omega)| \to 0, 
\text{ for $\PP$-a.e } \omega. 
\end{eqnarray*}
\end{proposition}

When $T_{\theta, f}$ is ergodic, if $\psi \in L^1(\mu \times \PP)$, we have ${\Cal I}_{\theta, f}(\psi)(x, \omega) = \int \psi \, d\mu \, d\PP$, 
for $\mu \times \PP \text {-a.e.} \, (x, \omega)$, and the centering  $\Cal I(\psi_s)(x, \omega)$ is given by the distribution function $F(s) = \mu(\varphi \leq s)$. 
In this case, for a.e. $x$, a Glivenko-Cantelli theorem with the usual centering holds 
for the empirical process sampled along the sequence $(z_n)$ given by $z_n = S_nf(x)$ (with a set of $\omega$'s of $\PP$-measure 1 depending on $x$). 

The lemma below shows, as it is known, that ergodicity of the cylinder map $\tilde T_f$ 
implies ergodicity of the skew map $T_{\theta, f}$. Let us sketch a proof.
\begin{lem} \label{ergCylind}
Suppose that the cocycle $(T, f)$ is recurrent and the map $\tilde T_f$ ergodic. If the action of $\Z^d$ by $\theta$ on $(\Omega, \PP)$ is ergodic, 
then $T_{\theta, f}$ is ergodic on $(X \times \Omega, \mu \times \PP)$.
\end{lem}
\proof: 
Let $\Phi$ be a $T_{\theta, f}$ invariant measurable function on $X \times \Omega$:
$$\Phi(Tx, \theta^{f(x)} \omega) = \Phi(x, \omega), \text{ for a.e. } (x, \omega).$$
For a.e. $x$, there is a set $\Omega_x^0$ of full $\PP$-measure in $\Omega$ such that $\Phi(Tx, \theta^{f(x)} \omega) = \Phi(x, \omega)$, 
for all $\omega \in \Omega_x^0$.
As $\Z^d$ is countable, for a.e. $x$, there is a set $\Omega_x$ of full measure such that 
$$\Phi(Tx, \theta^{f(x)} \theta^\el\omega) = \Phi(x, \theta^\el \omega), \text{ for all } \omega \in \Omega_x.$$
Let $\omega \in \Omega_x$. The function $\varphi_\omega(x, \el) := \Phi(x, \theta^\el \omega)$ on $X \times \Z^d$ is measurable, $\tilde T_f$-invariant:
\begin{eqnarray*}
\varphi_\omega(\tilde T_f(x, \el)) &=& \varphi_\omega(T x, \el + f(x)) = \Phi(T x, \theta^{\el + f(x)} \omega) \\
&=& \Phi(T x, \theta^{f(x)} \theta^\el \omega) =  \Phi(x, \theta^\el \omega) = \varphi_\omega(x, \el).
\end{eqnarray*}
It follows from the ergodicity of $\tilde T_f$ that there is a constant $c_\omega$ such that $\varphi_\omega(x, \el) = c_\omega$ for a.e. $x$.
Therefore $\Phi$ coincides a.e. with a function $\psi$ on $\Omega$ which is $\theta$-invariant, hence a constant by the assumption of ergodicity of the action of $\Z^d$ on 
$\Omega$. \eop

\vskip 3mm
With Fubini's argument, we get a Glivenko-Cantelli theorem for a.e. $x$, if we can show that the skew map $T_{\theta, f}$ is ergodic.

There are many examples cylinder flows $\tilde T_f$ which are shown to be ergodic in the literature and so providing examples via Lemma \ref{ergCylind}.
For instance, we can take for $T$ an irrational rotation and $f={\bf 1}_{[0,\frac12)}-{\bf 1}_{[\frac12, 1)}$. The cocycle ($T, f)$ is ergodic 
and the above version of Glivenko-Cantelli theorem applies for any stationary sequence $(X_k)$ (with a conditional distribution if the stationary sequence is not ergodic).
See also examples for which the skew map is ergodic in \cite{LLPVW02}. 

\vskip 3mm
\subsection{Discussion: universal sequences}

\ 

The weakness in the approach of the previous subsection for a sampled Glivenko-Cantelli theorem along ergodic sums $(S_k f(x), k \geq 0)$ is 
that it yields a set of $x$'s of $\mu$-measure 1 depending on the dynamical system $(\Omega, \PP, \theta)$ and on $\varphi$.
One can try to reinforce the statement by introducing a notion of ``universal property''.

In this direction, the LLN for sums sampled along ergodic sums is closely related in the following way to the random ergodic theorems 
which have been studied in several papers.

First, let us call ``universally good'' a sequence $(z_k)$ such that, for every dynamical system $(\Omega, \PP, \theta)$,
for every $\varphi \in L^1(\PP)$, the sequence $\frac1n \sum_{k=0}^{n-1} \varphi \circ \theta^{z_k}$ converges $\PP$-a.e.

We say that $(T, f)$ a  ``(pointwise) good averaging cocycle'' (or a universally representative sampling scheme) if, for $\mu\text{-a.e. } x$, 
the sequence $(S_k f(x))$ is universally good, i.e., for every dynamical system $(\Omega, \PP, \theta)$, for every $\varphi \in L^1(\PP)$, 
$\frac1n \sum_{k=0}^{n-1} \varphi \circ \theta^{S_k f(x)}$ converges $\PP$-a.e.

The definition of a ``mean good averaging cocycle'' is similar, changing the above convergence into convergence in $L^2(\PP)$-norm, 
for every $\varphi$ in $L^2(\PP)$.

A question which has been studied is to find mean or pointwise good averaging cocycles. 
In the first direction, examples and counterexamples of mean good averaging 1-dimensional cocycles are studied in \cite{LLPVW02}, 

For pointwise convergence, there are 1-dimensional examples given by cocycles with a drift. 
In \cite{LPWR94}, the following result is shown: the cocycle defined by a random walk with a moment of order 2 is a pointwise good averaging cocycle if and only if it is not centered.
Moreover it is shown that any ergodic integrable integer-valued stochastic process with nonzero mean is universally representative for bounded stationary processes. 
The proofs are based on the recurrence time theorem (\cite{BFKO89}).

Notice that a related, but different, notion can be introduced by restricting the dynamical system $(\Omega, \PP, \theta)$ 
to belong to a given class ${\Cal C}$ of dynamical systems.

Let us call ``pointwise good for a class $\cal C$ of dynamical systems'', a sequence $(z_k)$ such that, for every dynamical system $(\Omega, \PP, \theta)$
in the class $\cal C$, for every $\varphi \in L^1(\PP)$, $\lim_n \frac1n \sum_{k=0}^{n-1} \varphi \circ \theta^{z_k}= \int \varphi \, d \PP$, $\PP$-a.e.
There is a similar property for the mean convergence.

This can be also expressed for a class of random fields satisfying a condition on the decay of correlations.

For example, by Remark \ref{lebSpec}, every cocycle with values in $\Z^d$ which is not a coboundary is a mean good averaging cocycle 
for the stationary r.f.s on $\Z^d$ such that $\sum_\el |\langle U_\el , U_{\0}\rangle| < +\infty$.

If $(z_k)$ is pointwise universally good for a class $\cal C$, clearly we get the Glivenko-Cantelli property for any dynamical system $(\Omega, \PP, \theta)$ in $\cal C$
and every measurable function $\varphi$, i.e.:
\begin{eqnarray}
&\sup_s |\frac1n \, \sum_{k=0}^{n-1} \, {\bf 1}_{I_s} (\varphi(\theta^{z_k}\omega)) -\PP(\varphi \leq s)| \to 0, 
\text{ for $\PP$-a.e } \omega. \label{univ1}
\end{eqnarray}

As we see, there are two different approaches of the notion of universal sequences for a law of large numbers:
either we ask for a LLN along such a sequence for every dynamical system $(\Omega, \PP, \theta)$ and all functions in $L^1(\PP)$
or we fix a class of dynamical systems, or a class of functions  in $L^1(\PP)$. In the latter case, the condition on the sequence $(z_k)$ may be expressed in a quantitative way.
Let us give a known example and recall the proof.
\begin{proposition} Let $(z_k)$ be a strictly increasing sequence of positive integers. 
If the sequence satisfies: for a finite constant $C$, $z_k \leq C k, \forall k \geq 1$, then $(z_k)$ is a pointwise good averaging sequence 
for the class $\Cal C$ of dynamical systems $(\Omega, \PP, \theta)$ with Lebesgue spectrum.
\end{proposition}
\Proof There is a dense set of functions $\varphi \in L^1(\PP)$ such that
\begin{eqnarray}
\frac1n \sum_{k=0}^{n-1} \varphi(\theta^{z_k} \omega) \text{ converges } \PP\text{-a.e.} \label{convzk1}
\end{eqnarray}
Indeed, by the SLLN for orthogonal random variables, (\ref{convzk1}) is satisfied by $\varphi \in L^2(\PP)$ such that 
$\langle \varphi, \varphi \circ \theta^k\rangle = 0, \forall k$. The Lebesgue spectrum property implies that such functions span a dense linear space in 
$L^2(\PP)$, hence in $L^1(\PP)$.

Moreover, the space of functions $\varphi$ such that (\ref{convzk1}) holds is closed by the ergodic maximal lemma in view of the assumption on $(z_k)$.
Therefore (\ref{convzk1}) is satisfied by every $\varphi \in L^1(\PP)$. \eop

To finish, we recall the following example which shows that the behaviour may depend on the properties of the dynamical system $(\Omega, \PP, \theta)$ (cf. \cite{Co73}):

Let $(\Omega, \cal F, \PP)$ be the interval $[0, 1]$ endowed with the Borel $\sigma$-algebra and the Lebesgue measure and take $f = 1_{[0, \frac12]}$. 
Denote by $\cal T$ the class of invertible measure preserving transformations on this space. 
It can be shown that there are increasing sequences $(z_k)$ of positive integers satisfying the conditions of the previous proposition such that, for a dense $G_\delta$ 
of elements in $\cal T$ with continuous spectrum, the ergodic means of $f$ along $(z_k)$ do not converge $\PP$-a.e.

\end{document}